\documentclass[a4paper]{gtart}

\usepackage{psfrag, graphicx, subfigure, epsfig,color}

\usepackage{amsmath, amssymb, latexsym, euscript}

\usepackage{mathptmx, wrapfig}

\usepackage[small,bf]{caption}
\usepackage[inner=20mm, outer=20mm, textheight=235mm]{geometry}

\usepackage{booktabs}

\def\Vol{\operatorname{Vol}}
\def\lst{\operatorname{LST}}
\def\tri{\mathcal{T}}
\def\manifold{M}
\def\core{M^c}
\def\closure{\overline{M}}

\def\Z{\mathbb{Z}}
\def\R{\mathbb{R}}

\DeclareMathOperator{\Sym}{Sym}

\DeclareMathOperator{\nhd}{Nhd}

\def\Tee{\Delta_{\emptyset}}
\def\Ttt{\Delta_{\mathfrak{tt}}}
\def\Tqtt{\Delta_{\mathfrak{qtt}}}
\def\Tqq{\Delta_{\mathfrak{qq}}}
\def\Tqqq{\Delta_{\mathfrak{qqq}}}
\def\even{\mathfrak{e}}

\def\Tt{\Delta_{\mathfrak{t}}}
\def\Tq{\Delta_{\mathfrak{q}}}

\def\nee{n_{\emptyset}}
\def\ntt{n_{\mathfrak{tt}}}
\def\nqtt{n_{\mathfrak{qtt}}}
\def\nqq{n_{\mathfrak{qq}}}
\def\nqqq{n_{\mathfrak{qqq}}}

\theoremstyle{plain}
\newtheorem{theorem}{Theorem}
\newtheorem{lemma}[theorem]{Lemma}
\newtheorem{proposition}[theorem]{Proposition}
\newtheorem{corollary}[theorem]{Corollary}

\theoremstyle{definition}

\newtheorem*{definition*}{Definition}

\theoremstyle{remark}
\newtheorem{remark}[theorem]{Remark}

\numberwithin{equation}{section}

\makeatletter
\newtheorem*{rep@theorem}{\rep@title} \newcommand{\newreptheorem}[2]{%
\newenvironment{rep#1}[1]{%
\def\rep@title{\bf #2 \ref{##1} }%
\begin{rep@theorem} }%
{\end{rep@theorem} } }
\makeatother
\newreptheorem{theorem}{Theorem}

\usepackage[colorinlistoftodos]{todonotes}

\begin{document}

\title{On minimal ideal triangulations of cusped hyperbolic 3--manifolds}
\author{William Jaco, Hyam Rubinstein, Jonathan Spreer and Stephan Tillmann}

\begin{abstract}
Previous work of the authors studies minimal triangulations of closed 3--manifolds using a characterisation of low degree edges, embedded layered solid torus subcomplexes and 1--dimensional $\Z_2$--cohomology. The underlying blueprint is now used in the study of minimal ideal triangulations. As an application, it is shown that the monodromy ideal triangulations of the hyperbolic once-punctured torus bundles are minimal.
\end{abstract}

\makeshorttitle


\section{Introduction}

Thurston observed that ideal triangulations are a useful tool to study the geometry and topology of cusped hyperbolic 3--manifolds. Geometric triangulations are useful to study geometric properties of a manifold. Minimal triangulations, i.e.\thinspace topological ideal triangulations using the least number of ideal 3--simplices, are used in census enumeration, and as a platform to study the topology of the manifold using normal surface theory. In this paper we establish basic facts about minimal ideal triangulations of cusped hyperbolic 3--manifolds in analogy with our previous work for closed 3--manifolds~\cite{Jaco-Z2-2017,Jaco-minimal-2009, Jaco-Z2-2013}.

Throughout this paper, by a \emph{cusped hyperbolic 3--manifold} we mean an orientable non-compact 3--manifold $\manifold$ that admits a complete hyperbolic structure of finite volume. Such a manifold is known to be the interior of a compact, irreducible, $\partial$--irreducible, atoroidal, anannular 3--manifold $\overline{\manifold}$ with non-empty boundary a finite union of tori. If $\overline{\manifold}$ has exactly $n$ boundary components, we say that $\manifold$ is $n$--cusped.

Let $c(\manifold)$ be the minimal number of ideal tetrahedra in a (topological) ideal triangulation of the 
cusped hyperbolic 3--manifold $\manifold.$ In this paper, we use the \emph{topology} of $\manifold$ to study the anatomy of a minimal ideal triangulation, and with this derive lower bounds on the complexity, and apply it to describe some infinite families of minimal ideal triangulations.
Before stating our results, we describe previous work on minimal ideal triangulations using the \emph{geometry} of the manifold.


\subsection{Minimal ideal triangulations via geometry}

The canonical cell decompositions of cusped hyperbolic 3--manifolds due to Epstein and Penner~\cite{Epstein-euclidean-1988} are generically ideal triangulations. It follows from work of Gu\'eritaud \cite[Theorem 2.1.7 in Section 2.1.4]{GueritaudPhDThesis} that the canonical cell decompositions of hyperbolic 2--bridge link complements are ideal triangulations. However, these are not minimal in general. We thank Jessica Purcell and the anonymous referee for pointing us to these examples.

If four ideal tetrahedra in a geometric ideal triangulation form a non-convex ideal octahedron, then there are 4--4 moves on this triangulation that result in non-geometric ideal triangulations. Neil Hoffman verified that in the current censuses of cusped hyperbolic 3--manifolds up to 6 tetrahedra, there is always at least one minimal triangulation that is geometric, but that non-geometric minimal ideal triangulations appear through the presence of 4--4 moves as described above. In particular, it is currently not known whether there is always at least one minimal ideal triangulation that is geometric.

Let $v_3$ denote the volume of a regular ideal hyperbolic tetrahedron; this is approximately $1.0149.$ An argument due to Thurston~\cite{thurston78-lectures} shows that 
\begin{equation}\label{eq:volume bound}
c(\manifold) \ge \frac{\Vol(\manifold)}{v_3}.
\end{equation}
In particular, any geometric triangulation only involving regular ideal hyperbolic tetrahedra is minimal. The figure eight knot complement has such a triangulation. By taking finite sheeted  coverings this leads to infinitely many minimal triangulations that are geometric and for which the above inequality is an equality. A census of such manifolds decomposed into at most 25 ideal regular ideal hyperbolic tetrahedra was given by Fominykh, Garoufalidis, Goerner, Tarkaev and Vesnin~\cite{Fominykh-census-2016}. 

In general the gap in inequality~(\ref{eq:volume bound}) can be arbitrarily large. To see this neither requires explicit minimal triangulations nor computation of volumes: Thurston's theory of hyperbolic Dehn surgery implies that for any $n\ge 1$ there are sequences of $n$--cusped hyperbolic 3--manifolds with bounded volume and whose complexities are an unbounded sequence.

An equivalent approach to complexity is via Matveev's theory of special spines. From this point of view,
Petronio and Vesnin~\cite{Petronio-two-sided-2009} give a lower bound on complexity using a volume estimate.
Ishikawa and Nemoto~\cite{Ishikawa-construction-2016} combined this with an upper bound on the complexity of 2--bridge links and determined the complexity of an infinite family of 2--bridge links. These ideal triangulations were described by Sakuma and Weeks~\cite{Sakuma-examples-1995} and proved to be canonical by Gu\'eritaud \cite[Theorem 2.1.7 in Section 2.1.4]{GueritaudPhDThesis}.
Independently, Akiyoshi, Sakuma, Wada and Yamashita have announced a proof of this statement~\cite{Akiyoshi-punctured-2007}. A proof of this fact for the fibered 2--bridge links was given by Sakata~\cite{Sakata-canonical-2015}.


\subsection{Minimal ideal triangulations via topology}

 In this paper, we are first concerned with the anatomy of a minimal ideal triangulation.
In particular, in \S\ref{sec:anatomy} we characterise edges of low degree, and analyse subcomplexes that we call \emph{maximal layered $\partial$--punctured solid tori.} These are solid tori with a point missing on the boundary, and imbued with a special ideal triangulation.
These subcomplexes were exhibited by 
Gu\'eritaud and Schleimer~\cite{Gueritaud-canonical-2010} in canonical decompositions of cusped hyperbolic 3--manifolds satisfying certain genericity hypotheses, and are thus a natural part of the structure of an ideal triangulation to analyse from a geometrical viewpoint. Our methods show that they also have an important topological role to play.
In \S\ref{sec:normsurfs} we describe normal surfaces representing $\Z_2$--homology classes in ideal triangulations. We use this in \S\ref{sec:quadsurfs} to prove our main result Theorem~\ref{thm:sumofnorms}, which is analogous to the main result of \cite{Jaco-Z2-2013}. Before we can state it, we need to introduce some extra definitions.

Let $\manifold$ be a cusped hyperbolic 3--manifold, and let
$S$ be an embedded closed surface representing a given $c \in H_2(\manifold;\Z_2).$ 
We wish to emphasise that we do not work with $H_2(\overline{\manifold}, \partial \overline{\manifold} ;\Z_2).$
An analogue of Thurston's norm \cite{Thurston-norm-1986} is defined in \cite{Jaco-Z2-2013} as follows. If $S$ is connected, let $\chi_{-}(S) = \max \{ 0,-\chi(S)\},$ and otherwise let
\[\chi_{-}(S) = \sum_{S_i\subset S} \max \{ 0,-\chi(S_i)\},\]
where the sum is taken over all connected components of $S.$ Note that $S_i$ is not necessarily orientable. Define:
\[|| \ c \ || = \min \{ \chi_{-}(S) \mid [S]= c\}.\]
The surface $S$ representing a class $c \in H_2(\manifold;\Z_2)$ is said to be \emph{$\Z_2$--taut} if no component of $S$ is a sphere or torus and $\chi(S) = -|| \ c \ ||.$ Note that $\manifold$ contains no projective plane or Klein bottle.
As in \cite{Thurston-norm-1986}, one observes that every component of a $\Z_2$--taut surface is non-separating and geometrically incompressible.

\begin{theorem}
  \label{thm:sumofnorms}
  Let $\tri$ be an ideal triangulation of the cusped hyperbolic $3$--manifold $\manifold.$ If $H \le H_2 (\manifold,\Z_2)$ is a rank $2$ subgroup, then 
  $$|\tri| \geq \sum \limits_{0 \neq c \in H} || \ c\ ||. $$
   Moreover, in the case of equality the triangulation is minimal, each canonical normal representative of a non-zero element in $H\le H_2 (\manifold,\Z_2)$ is taut and meets each tetrahedron in a quadrilateral disc, and the number of tetrahedra in the triangulation is even.
\end{theorem}

As an application, we obtain the following statement.

\begin{theorem}
  \label{thm:torusbundles}
  Monodromy ideal triangulations of once-punctured torus bundles are minimal.
\end{theorem}

Monodromy ideal triangulations of hyperbolic once-punctured torus bundles are described by Floyd and Hatcher~\cite{Floyd-incompressible-1982} and attributed to Thurston. These triangulations are canonically determined by the monodromy of the bundle acting on the Farey graph; they are an early special case of Agol's veering triangulations. The fact that the monodromy ideal triangulation of each hyperbolic once-punctured torus bundle is canonical, and hence also geometric, is due to Akiyoshi~\cite{Akiyoshi99CanonicityLayeredTorusBundles}, Lackenby~\cite{Lackenby-canonical-2003} and Gu\'eritaud \cite{GueritaudPhDThesis}.

Such a triangulation may contain edges of arbitrarily high degree, hence arbitrarily many ideal hyperbolic tetrahedra of arbitrarily small volume, a fact which is also highlighted by explicit upper bounds on the volume of torus bundles in \cite[Theorem B.1]{Gueritaud-canonical-2006} and in \cite[Corollary 2.4]{Agol03SmallLargeGenus}. Thus, our family contains explicit minimal triangulations where the gap between complexity and volume can be arbitrarily large. 

The monodromy ideal triangulation $\tri$ of a hyperbolic once-punctured torus bundle $\manifold$ with the property that $H_2(\manifold,\Z)$ contains a Klein 4--group $H$ satisfies the equality
$$|\tri| = \sum \limits_{0 \neq c \in H} || \ c\ ||.$$
There are other once-cusped hyperbolic 3--manifolds, for which this equality is achieved. These are not fibred, and are described in \S\ref{subsec:more examples}.
It remains an open problem to \emph{classify all triangulations} for which the lower bound in Theorem~\ref{thm:sumofnorms} is achieved. Such a classification is given in \cite{Jaco-Z2-2017} for our lower bound on the complexity for closed atoroidal 3--manifolds.

\subsection{Other classes of hyperbolic 3--manifolds}

We remark that for closed hyperbolic 3--manifolds, currently no infinite families of minimal triangulations are known. In the case of compact manifolds with non-empty boundary, lower bounds on complexity are known through work of \cite{Bucher-simplicial-2015} and \cite{Jaco-bounds-2016}, where it is independently shown that they are attained by $S_g\times[0,1],$ where $S_g$ is a closed orientable surface of genus $g\ge 1.$

There are infinite families of minimal triangulations of manifolds with a totally geodesic boundary component of genus $g\ge 2$ and some cusps due to Frigerio, Martelli and Petronio~\cite{Frigerio-Dehn-2003}. These triangulations can be thought of as subdivisions into partially truncated tetrahedra; or alternatively as ideal triangulations with one ideal vertex having link a surface of genus $g$ and all others having link a torus. We remark that our methods and results do not generalise to this setting of higher genus vertex links. The main reason is that there is less control over the anatomy of these triangulations (cf. Remark~\ref{rem:fmp}).


\textbf{Acknowledgements.}
The authors thank the anonymous referee for an excellent job, pointing out many additional sources and inaccuracies in our original arguments. Their comments helped us to significantly improve this article.

This research was supported through the programme ``Research in Pairs'' by the Mathematisches Forschungsinstitut Oberwolfach in 2017. The authors would like to thank the staff at MFO for an excellent collaboration environment.
Jaco is partially supported by NSF grant DMS-1308767 and the Grayce B. Kerr Foundation.
Spreer is partially supported by the Einstein Foundation (project ``Einstein Visiting Fellow Santos''). 
Research of Rubinstein and Tillmann is supported in part under the Australian Research Council's Discovery funding scheme (project number DP160104502). Tillmann thanks the DFG Collaborative Center SFB/TRR 109 at TU Berlin, where parts of this work have been carried out, for its hospitality.


\section{The anatomy of a minimal ideal triangulation}
\label{sec:anatomy}


\subsection{Triangulations and pseudo-manifolds}
\label{sec:Triangulations}

Let $\widetilde{\Delta}$ be a finite union of pairwise disjoint Euclidean 3--simplices with the standard simplicial structure. Every $k$--simplex $\tau$ in $\widetilde{\Delta}$ is contained in a unique 3--simplex $\sigma_\tau.$ A 2--simplex in $\widetilde{\Delta}$ is termed a \emph{face}.

Let $\Phi$ be a family of affine isomorphisms pairing the faces in $\widetilde{\Delta},$ with the properties that $\varphi \in \Phi$ if and only if $\varphi^{-1}\in \Phi,$ and every face is the domain of a unique element of $\Phi.$ The elements of $\Phi$ are termed \emph{face pairings}.

Consider the quotient space $\widehat{\manifold} = \widetilde{\Delta}/\Phi$ with the quotient topology, and denote the quotient map $p\co \widetilde{\Delta} \to \widehat{\manifold}.$ We make the additional assumption on $\Phi$ that for every $k$--simplex $\tau$ in $\widetilde{\Delta}$ the restriction of $p$ to the interior of $\tau$ is injective. Then
the set of non-manifold points of $\widehat{\manifold}$ is contained in the 0--skeleton, and in this case $\widehat{\manifold}$ is called a \emph{closed 3--dimensional pseudo-manifold.} (See Seifert-Threfall~\cite{Seifert-textbook-1980}.) We also always assume that $\widehat{\manifold}$ is connected.
The triple $\tri = ( \widetilde{\Delta}, \Phi, p)$ is called a \emph{(singular) triangulation} of $\widehat{\manifold}.$

Let $\widehat{\manifold}^{(k)}= \{ p(\tau^k) \mid \tau^k \subseteq \widetilde{\Delta}\}$ denote the set of images of the $k$--simplices 
of $\widetilde{\Delta}$ under the projection map. Then the elements of $\widehat{\manifold}^{(k)}$ are precisely the equivalence classes of $k$--simplices in $\widetilde{\Delta}.$
The elements of $\widehat{\manifold}^{(0)}$ are termed the \emph{vertices} of $\widehat{\manifold}$ and the elements of $\widehat{\manifold}^{(1)}$ are the \emph{edges}. 
The triangulation is a \emph{$k$--vertex triangulation} if $\widehat{\manifold}^{(0)}$ has size $k.$ 


\subsection{Edge paths and edge loops}
\label{sec:Edge paths and edge loops}

To each edge $e$ in $\widehat{\manifold}^{(1)}$ we associate a continuous path $\gamma_e \co [0,1]\to \widehat{\manifold}$ with endpoints in $\widehat{\manifold}^{(0)}$ via the composition of maps $[0,1] \to \widetilde{\Delta} \to \widehat{\manifold},$ where the first map is an affine map onto a 1--simplex in the equivalence class and the second is the quotient map. We call $\gamma_e$ an \emph{edge path}. An edge path is termed an \emph{edge loop} if its ends coincide. For instance, in a 1--vertex triangulation every edge path is an edge loop.

We often abbreviate ``edge path" or ``edge loop" to ``edge", and we denote the edge path $\gamma_e$ with reversed orientation by $-\gamma_e.$ The notions of interest in this paper are independent of the parametrisation.


\subsection{Ideal triangulations}
\label{sec:Ideal triangulations}

Let $\manifold = \widehat{\manifold}\setminus \widehat{\manifold}^{(0)}.$ Then $\manifold$ is a topologically finite, non-compact 3--manifold, and the pseudo-manifold $\widehat{\manifold}$ is referred to as the \emph{end-compactification} of $\manifold ,$ the elements of $\widehat{\manifold}^{(0)}$ as the \emph{ideal vertices} of $\manifold$ and the elements of $\widehat{\manifold}^{(1)}$ as the \emph{ideal edges} of $\manifold .$ 
We also refer to the triangulation $\tri = ( \widetilde{\Delta}, \Phi, p)$ of $\widehat{\manifold}$ as an \emph{ideal (singular) triangulation} of $\manifold .$

The adjective \emph{singular} is usually omitted: we do not need to distinguish between simplicial and singular triangulations. If $\widehat{\manifold}$ is a closed 3--manifold, we may write $M = \widehat{\manifold},$ and hope this does not cause any confusion. We suppress the notation $\tri = ( \widetilde{\Delta}, \Phi, p)$ and simply say that $\widehat{\manifold}$ is triangulated or $\manifold = \widehat{\manifold}\setminus \widehat{\manifold}^{(0)}$ is ideally triangulated.


The following result is implicit in Matveev \cite{Matveev-algorithmic-2007} and shows that every topologically finite 3--manifold arises in this way from a closed pseudo-manifold. A stronger version for compact orientable irreducible an-annular 3--manifolds with $\partial \manifold$ a non-empty collection of incompressible tori, avoiding the use of spines can be found in \cite[Theorem 1]{Lackenby00TautIdeal}. 

\begin{proposition}[Topologically finite manifolds have ideal triangulations]\label{pro:ideal triangulation exists}
If $\manifold$ is the interior of a compact 3--manifold $\overline{\manifold}$ with non-empty boundary, then $\manifold$ admits an ideal triangulation. The ideal vertices of the ideal triangulation are in one-to-one correspondence with the boundary components of $\overline{\manifold}.$
\end{proposition}

\begin{proof}
The statement is implied by the following results in \cite{Matveev-algorithmic-2007}. Theorem 1.1.13 due to Casler asserts that $\overline{\manifold}$ possesses a special spine $\Sigma;$ it has the property that $\overline{\manifold}$ is homeomorphic to a regular neighbourhood of $\Sigma$ in $\manifold .$ Theorem 1.1.26 implies that $\Sigma$ is dual to an ideal triangulation of $\manifold$ with the property that the ideal vertices are in one-to-one correspondence with the boundary components of $\overline{\manifold}.$
\end{proof}


\subsection{Layered solid tori}
\label{sec:lst}

A {\em layered solid torus} $\lst$ is a triangulation of a solid torus iteratively constructed in the following way: Start with a one-triangle M\"obius strip as shown in Figure~\ref{fig:faces}(c). Then attach additional tetrahedra without a twist along pairs of (unpaired) faces of the existing complex sharing an edge $e$. This operation is called a {\em layering} on $e$, where the first layering onto the M\"obius strip must always go on the interior edge of the M\"obius strip, see Figure~\ref{fig:layering}. Layered solid tori are close cousins of the monodromy ideal triangulations of once-punctured torus bundles defined in Section~\ref{sec:mit}, since both types of triangulations are obtained from layering tetrahedra onto an existing complex. See \cite{Jaco-layered-2006} for a comprehensive introduction into layered solid tori and, more generally, layered triangulations. Here, we summarise material of \cite[Section 2.3]{Jaco-minimal-2009}.

\begin{figure}[htb]
  \centering{\includegraphics[width=6cm]{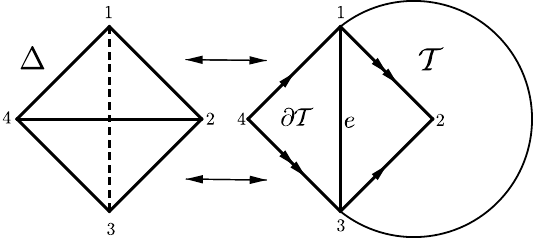}}
    \caption{Layering tetrahedron $\Delta$ to triangulation $\tri$ on edge $e$. Identifications of the triangles are indicated by the vertex labels. \label{fig:layering}}
\end{figure}

By construction, the boundary of a layered solid torus containing at least one tetrahedron is always the standard triangulation of a torus with one vertex, two triangles, and three edges. Layered solid tori are distinguished by how often their meridional disc intersects their three boundary edges geometrically and we write $\lst (a,b,c)$ for a layered solid torus with meridional disc intersecting the boundary edges $a$, $b$ and $c$ times respectively. Assuming $a\leq b\leq c$ it is straightforward to see that $a+b = c$ and that a layered solid torus $\lst (a,b,c)$ can be turned into a layered solid torus of type $\lst (a,b,b-a)$, $\lst (a,c,a+c)$ or $\lst (b,c,b+c)$ by layering an additional tetrahedron on the edge with label $c$, $b$, or $a$ respectively.  

There exists only one layered solid torus with one tetrahedron, $\lst(1,2,3)$. Hence, by construction, every layered solid torus with at least one tetrahedron contains this layered solid torus, termed its {\em core tetrahedron}.

Let $\lst$ be a layered solid torus occurring as a subcomplex of some triangulation $\tri$ of a $3$-manifold. We say that $\lst$ is \emph{maximal with respect to $\tri$} if it is not strictly contained in any other layered solid torus subcomplex in $\tri.$ By a \emph{layered $\partial$--punctured solid torus}, we mean a \emph{layered solid torus} with its vertex removed. Whenever we talk about the punctured version of a layered solid torus $\lst (a,b,c)$ we denote it by $\lst^{\star} (a,b,c)$ to emphasise the difference. The notion of a layered $\partial$--punctured solid torus is very convenient when discussing the combinatorial structure of an ideal triangulation. For instance, they play a role in \cite{Gueritaud-canonical-2010}.


\subsection{Properties of minimal triangulations}

An ideal triangulation of a topologically finite 3--manifold $\manifold$ is \emph{minimal} if it uses the smallest number of simplices in an ideal triangulation of $\manifold .$ The number of ideal simplices in a minimal ideal triangulation of $\manifold$ is denoted $c(\manifold)$ and termed the \emph{complexity} of $\manifold .$

We assume that the reader is familiar with normal surface theory. We refer to \cite{Tillmann-normal-2008} for a thorough introduction, but remark that we only need to study \emph{closed} normal surfaces in ideal triangulations.

An ideal triangulation of $\manifold$ is \emph{0--efficient} if no normal surface is a 2--sphere.  An ideal triangulation of $\manifold$ is \emph{$\partial$--efficient} if it has the property that any closed normal surface that is isotopic to a vertex linking surface is \emph{normally isotopic} to that vertex-linking surface. A normal isotopy is an isotopy of $\manifold$ that preserves each simplex of each dimension. 

\begin{theorem}\label{lem:bdEfficient} 
Suppose $\manifold$ is the interior of a compact, irreducible, $\partial$--irreducible anannular 3--manifold. Then each minimal ideal triangulation of $\manifold$ is 0--efficient and $\partial$--efficient.
\end{theorem}

\begin{proof}
It is shown in \cite[Corollary 7.3]{Jaco-0-efficient-2003} that each minimal triangulation of $\manifold$ is 0--efficient. 
In \cite[Theorem 4.7]{Jaco-annular-2011}, the existence of a $\partial$--efficient triangulation is proven by a crushing method. Since crushing reduces the number of tetrahedra, this means that a triangulation that fails to be $\partial$--efficient is not minimal.
\end{proof}

We now give a definition of what we mean by an ideal edge to be homotopic or isotopic \emph{into the boundary}. The key is that intermediate paths in the homotopy are not allowed to pass through the ideal vertices, except for their endpoints. The reader is reminded than an ideal edge of $\manifold$ is an edge of $\widehat{\manifold}.$ An expanded discussion of the following material can be found in \cite{Hodgson-triangulations-2015}.

A path homotopy (resp. isotopy)\footnote{i.e. a homotopy (resp. isotopy) keeping endpoints fixed}
 $H\co [0,1] \times [0,1]\to \widehat{\manifold}$ between two edge paths is \emph{admissible} if $H(\; (0,1)\times [0,1]\;) \subset \manifold,$ where the first factor parametrises the paths. The edge path $\gamma$ is \emph{admissibly null-homotopic (resp. null-isotopic)} if there is a path homotopy (resp. isotopy) $H\co [0,1] \times [0,1]\to \widehat{\manifold}$ with $H(x, 0) = \gamma(x),$ $H(x, 1) = \gamma(0)$ for all $x\in [0,1]$ and $H(\; (0,1)\times [0,1)\;) \subset \manifold.$

We say that an ideal edge $e$ of $\manifold$ is \emph{homotopic (resp. isotopic) into the boundary} if it is admissibly null-homotopic (resp. null-isotopic) in $\widehat{\manifold}.$

\begin{corollary}\label{cor-min-bddry-eff}
Suppose $\manifold$ is the interior of a compact, irreducible, $\partial$--irreducible anannular 3--manifold other than a 3--ball, and $\tri$ is a minimal ideal triangulation of $\manifold .$ Then no ideal edge is isotopic into the boundary. In particular, no edge in $\widehat{\manifold}$ bounds an embedded disc in $\widehat{\manifold}.$
\end{corollary}

\begin{proof}
Suppose the edge $e$ is isotopic into the boundary. Choose a compact core $\core$ of $\manifold$ with the property that each boundary component of $\core$ is a vertex linking surface. 
There exists an embedded disc $D$ in $\core$ whose boundary consists of $e \cap \core$ and a subarc on a boundary component $B$ of $\core.$ The boundary of a regular neighbourhood of $B\cup D$ has two components; one is a vertex linking surface normally isotopic to $B$ and the other is a surface $S$ isotopic to $B$ and disjoint from $e.$ Since $S$ is incompressible and not a 2--sphere and $\manifold$ is irreducible, it follows that there is a normal surface isotopic but not normally isotopic to $B.$ Since the triangulation is $\partial$--efficient according to Theorem~\ref{lem:bdEfficient}, it follows that $\manifold$ is homeomorphic with $S \times (0,1).$ However, this contradicts our topological hypotheses on $\manifold .$
\end{proof}

Note that the above result does not apply to the unique minimal ideal triangulation $\tri$ of the 3--ball. This is the 1--tetrahedron, 1--vertex triangulation of $S^3$ with the vertex removed. The triangulation of $S^3$ has precisely two edges; one is a trefoil knot in $S^3$ and the other the unknot in $S^3.$ The former gives an ideal edge in $\tri$ that is only homotopic into the boundary, whilst the latter is isotopic into the boundary. The above proofs breaks down for the 3--ball because the surface $S$ in the proof shrinks to a sphere contained in a tetrahedron.

We now specialise to the class of manifolds of interest, namely cusped hyperbolic 3--manifolds, where more specific results can be shown about their minimal triangulations. In order to limit the cases to consider, we not only assume that all \emph{vertex links are of Euler characteristic zero}, but also that the manifolds are \emph{orientable}. 


\subsection{Faces in minimal ideal triangulations of cusped hyperbolic 3--manifolds}

\captionsetup[subfigure]{labelformat=empty}
\begin{figure}[h]
  \begin{center}
    \subfigure[triangle]{\includegraphics[height=2.1cm]{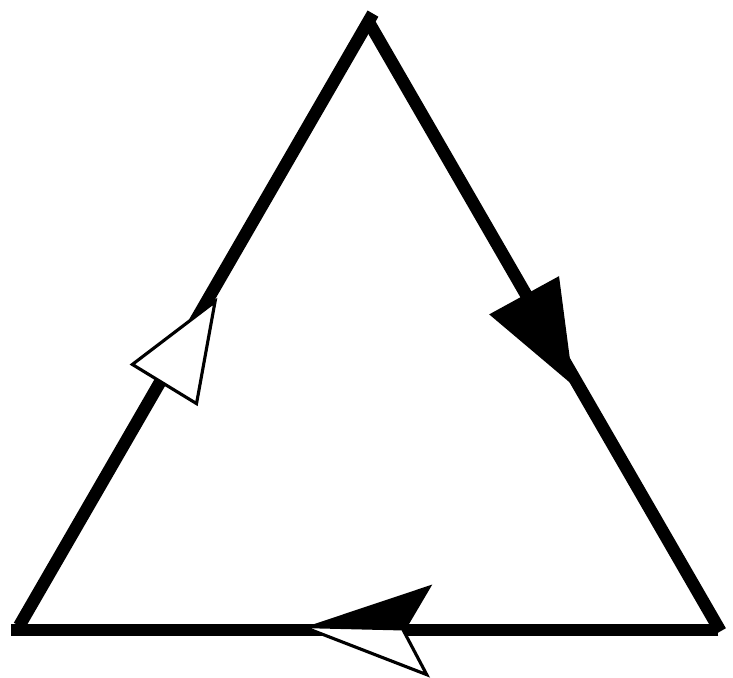}}
    \qquad
     \subfigure[cone]{\includegraphics[height=2.1cm]{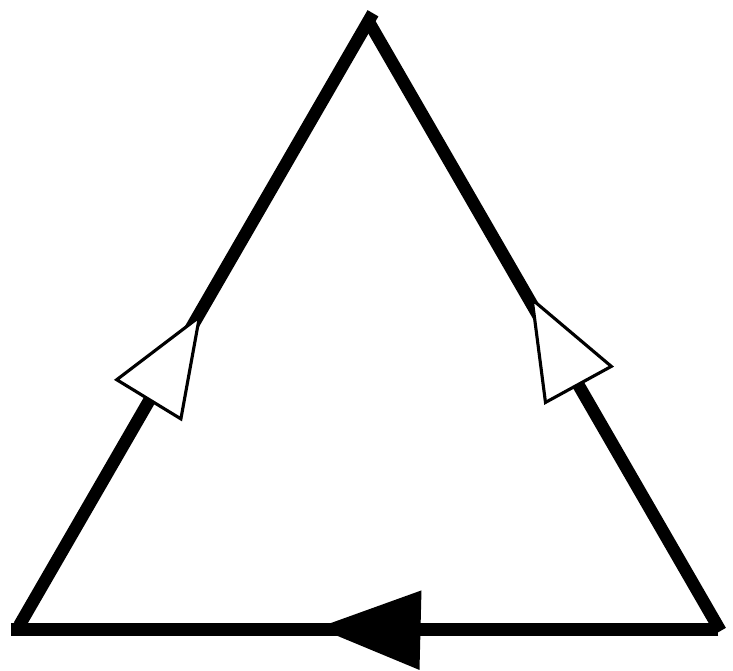}}
  \qquad
     \subfigure[M\"obius]{\includegraphics[height=2.1cm]{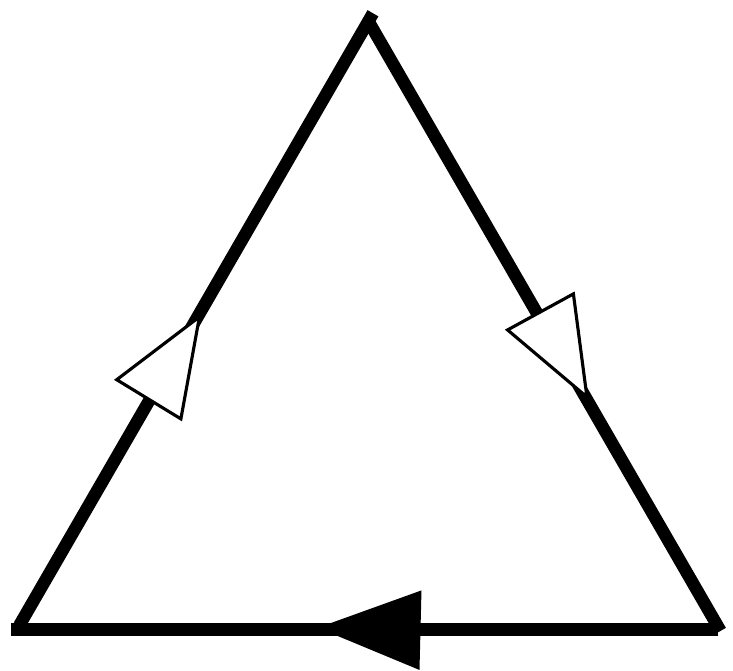}}
  \qquad
     \subfigure[3--fold]{\includegraphics[height=2.1cm]{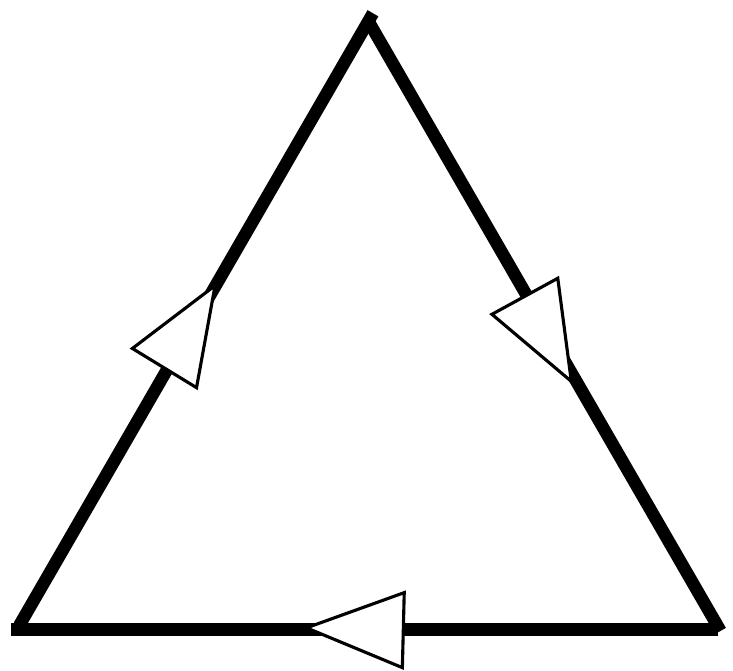}}
     \qquad
     \subfigure[dunce]{\includegraphics[height=2.1cm]{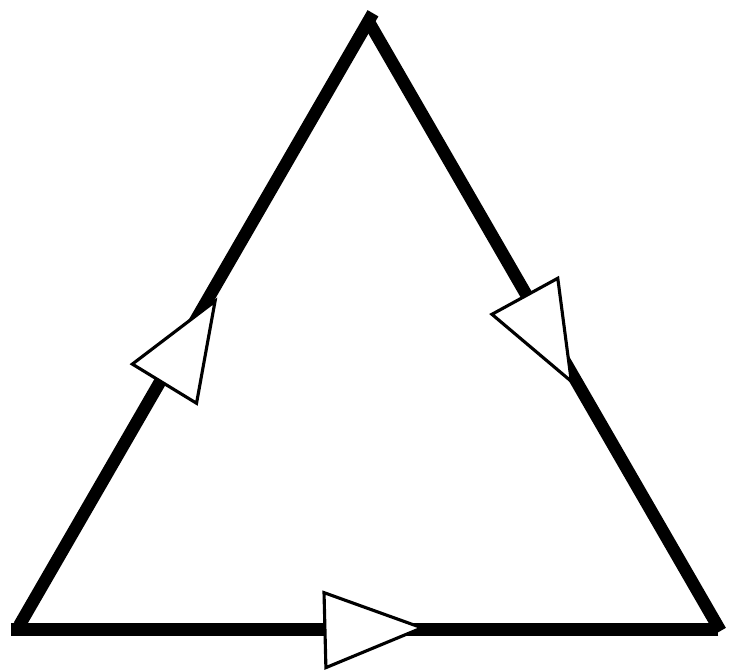}}
          \end{center}
   \caption{Types of faces. The M\"obius, 3--fold and dunce faces only have one vertex due to the edge identifications. A cone face may have one or two vertices and a triangle face may have up to 3 vertices.}
   \label{fig:faces}
\end{figure}
\captionsetup[subfigure]{labelformat=default}

An ideal triangle in an ideal triangulation of a cusped hyperbolic 3--manifold $\manifold$ may have some of its edges identified. 
There are eight possible types of triangular faces in the triangulation of the end-compactification $\widehat{\manifold},$ depending on how vertices or edges are identified (see Figure 1 in \cite{Jaco-0-efficient-2003}). Ignoring possible identifications of vertices as in \cite{Luo-combinatorial-2017}, this reduces to five types in $\manifold$ based on edge identifications only (see Figure~\ref{fig:faces}). We name the faces in $\manifold$ according to their topology in $\widehat{\manifold}.$ A face with no edge-identifications is termed a \emph{triangle face}; it is a 2--simplex, possibly with some or all of its vertices identified. If a pair of edges is identified, the face is either a cone (possibly with the tip identified with the vertex on the boundary) and called a \emph{cone face} or it is a M\"obius band and called a \emph{M\"obius face}. In a M\"obius face, we distinguish the \emph{boundary edge} and the \emph{core edge}. If all three edges are identified, the face is either a 3--fold or a dunce hat, and called a \emph{3--fold face} or \emph{dunce face} respectively.

For minimal ideal triangulations, we can prove the following statement.

\begin{theorem}
\label{thm:no3fold}
Let $\tri$ be a minimal ideal triangulation of a cusped hyperbolic 3--manifold $\manifold.$ Then there is no 3--fold face and there is no dunce face.
\end{theorem}

\begin{proof}
Choose a compact core $\core$ of $\manifold$ with the property that each boundary component of $\core$ is a vertex linking surface. 
Let $t$ be a triangle of $\tri.$ 

  {\bf Case 1: $t$ is a 3--fold face.} Refer to Figure~\ref{fig:threefold}.
  Let $e$ be the edge of $t,$ $N(e)$ a small regular neighbourhood of $e$ in $\manifold$ and $N^c(e) = N(e) \cap \core.$
    The frontier of $N^c(e)$ in $\core$ is an annulus; we call this the \emph{exchange annulus around $e.$} Let $N(t)$ be a regular neighbourhood of $t$ in $\manifold .$ Then the boundary of $N^c(t) = N(t) \cap \core$ has a natural cell decomposition into two hexagons corresponding to the two sides of $t$ and three quadrilaterals that are subsurfaces of the exchange annulus around $e.$ The edges of the hexagons through the interior of $\core$ are termed the \emph{long edges}. Seen with respect to $(t\cap \core) \setminus N(e)$ the long edges of the hexagons have a natural labelling $1, 2, 3$ in cyclic order. See Figure~\ref{fig:threefold} on the right for a drawing of $N^c(t),$ in particular the two hexagons and their long edges.
     
Give $t$ a transverse orientation. This allows us to refer to one of the hexagons as the \emph{top} hexagon and to the other as the \emph{bottom} hexagon. The intersection $t \cap N^c(e)$ is a tripod times $[0,1].$ Figure~\ref{fig:threefold} on the left depicts a cross-section of the tripod times $[0,1]$ at some $x \in (0,1).$ The transverse orientation of $t$ is in the same direction at the legs of the tripod (either all clockwise or all anti-clockwise around the vertex of the tripod). It follows that the quadrilaterals join each long edge of the top hexagon to a long edge of the bottom hexagon. Moreover, on the labels this induces a 3--cycle. Thus, modulo orientation, the identification must be as shown in Figure~\ref{fig:threefold} (both on the left and on the right).
  
After collapsing the sections of the exchange annulus to the long edges of the hexagons, we obtain a connected, orientable, bounded surface $\mathcal{S}$ decomposed into two hexagons, nine edges and six vertices -- and thus of Euler characteristic $-1.$ As can be deduced from the vertex labels, the surface has three boundary components and thus must be a 3--punctured sphere. 

\begin{figure}[htb]
 \begin{center}
 \includegraphics[width=\textwidth]{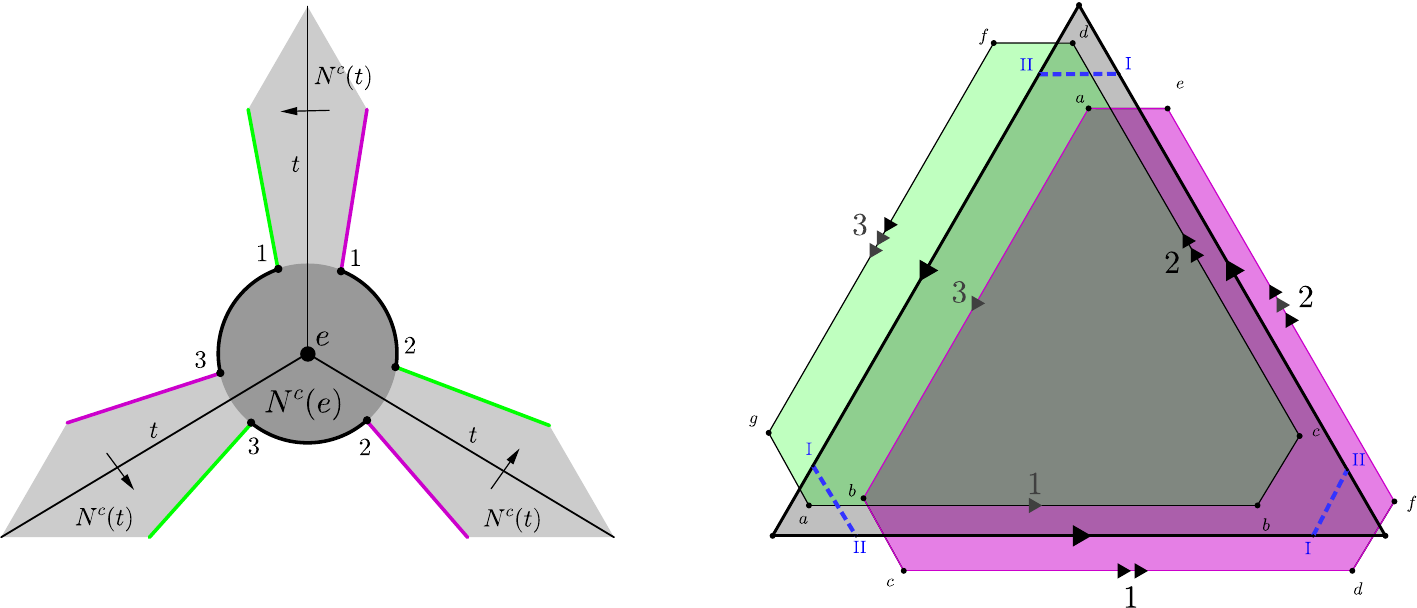}
 \caption{Left: Cross-section through the neighbourhood $N^c (t)$ of an ideal 3--fold $t$ near its boundary edge $e.$ The cross-section through the top and bottom hexagon of the right picture appear as green and purple lines respectively. Their endpoints mark the cross-section through the long edges of the respective hexagons. Labels coincide with the ones on the right. Right: Truncated boundary $N^c(t)$ of regular neighbourhood of the 3--fold face $t$ (drawn in the centre in black). Its boundary surface is connected, of Euler characteristic $-1$ with three boundary components (each containing two edges, as indicated by the edge labels) and thus a 3--punctured sphere. The dotted blue lines with labels represent the intersection of $t$ with the boundary torus $T_0$ -- a triple arc between nodes $1$ and $2,$ or the $\Theta$--graph.}
 \label{fig:threefold}
 \end{center}
\end{figure}

  Because $t$ is a 3--fold, the intersection of $t$ with a boundary torus $T_0 \subset \partial \core$ consists of two nodes and three (normal) arcs, all running between the two nodes, also known as the {\em $\Theta$--graph}, see the dotted arcs in Figure~\ref{fig:threefold}. Since $\mathcal{S}$ has three boundary components meeting $T_0$ along the boundary of a neighbourhood of $t \cap T_0,$ the complement of the embedding of the $\Theta$--graph into $T_0$ must have three boundary components. 

If no two arcs form a separating (i.e., inessential) curve on $T_0,$ then the $\Theta$--graph forms a spine of $T_0.$ In particular, its complement is a disc with a single boundary component, a contradiction. Hence, two arcs must bound a disc in $T_0.$ The remaining arc now either runs parallel to the first two, or completes each of the other two to an essential curve on $T_0.$ In both cases the complement has three boundary components. Both embeddings are shown in Figure~\ref{fig:embeddings_1}. 

Assume that we are in the latter situation and the $\Theta$--graph contains an essential curve on $T_0$. Since the embedding features two parallel arcs, one of the boundary curves of $\mathcal{S}$ bounds a disc in $T_0$. Pasting in this disc and pushing it into the interior of the manifold yields a properly embedded annulus which, because $\manifold$ is hyperbolic, must be compressible or boundary parallel. If it is compressible, then $T_0$ is compressible. We conclude that $\mathcal{S}$ must be boundary parallel.

Hence assume that the $\Theta$--graph consists of three parallel arcs and is thus contained in a disc of $T_0$. This time two of the boundary curves of $\mathcal{S}$ bound discs in $T_0$ and pasting them in and pushing them into the manifold yields a properly embedded disc, and we conclude again that $\mathcal{S}$ must be boundary parallel.

In both cases, uniting $\mathcal{S}$ with the complement of $N(t \cap T_0)$ in $T_0$ thus produces a torus 
$$T_0 \setminus N(t \cap T_0) \cup \mathcal{S}$$
boundary parallel to $T_0$ and disjoint from $e$ (which is itself incident to $T_0$). It follows that normalising this boundary parallel torus cannot result in a surface normally isotopic to $T_0$ and thus $\tri$ cannot be boundary efficient. This is impossible since we assume that $\tri$ is minimal and thus boundary efficient by Theorem~\ref{lem:bdEfficient}. 

\begin{figure}[hbt]
 \begin{center}
 \includegraphics[width=\textwidth]{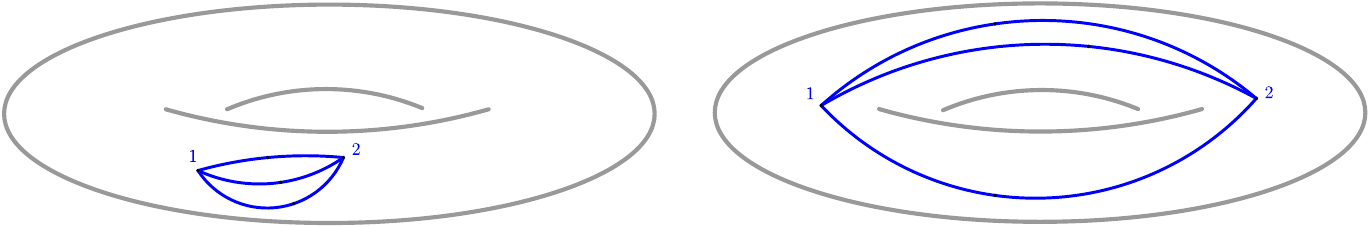}
 \caption{Two embeddings of the intersection of the 3--fold face $t$ with torus boundary component $T_0.$ Left: trivial embedding contained in a disc of $T_0;$ in this case the boundary parallel torus is made up of a once-punctured torus and two discs that are subsurfaces of $T_0$ and the 3--punctured sphere $\mathcal{S}.$ Right: two parallel essential loops; in this case the boundary parallel torus is made of a disc and an annulus that are subsurfaces of $T_0$ and the 3--punctured sphere $\mathcal{S}.$}
 \label{fig:embeddings_1}
 \end{center}
\end{figure}

{\bf Case 2: $t$ is a dunce face.} Refer to Figure~\ref{fig:duncehat}. Let $e$ be the edge of $t.$ We use the notation and set-up as in the previous case. The intersection of $t$ with $N^c(e)$ is again a tripod times an interval (see Figure~\ref{fig:duncehat} on the left for a cross-section). However, now the transverse orientations do not all agree: two are in one direction and the last is in the opposite direction. In particular, this implies that each of the hexagons has two of its long edges connected by a quadrilateral on the exchange annulus, and exactly one long edge of the top hexagon is joined to a long edge of the bottom hexagon (and these edges do not share the same label). It follows that, up to symmetry, there is only a single way to identify the long edges of the hexagons, shown in Figure~\ref{fig:duncehat} (both on the left and on the right). 

\begin{figure}[htb]
 \begin{center}
 \includegraphics[width=\textwidth]{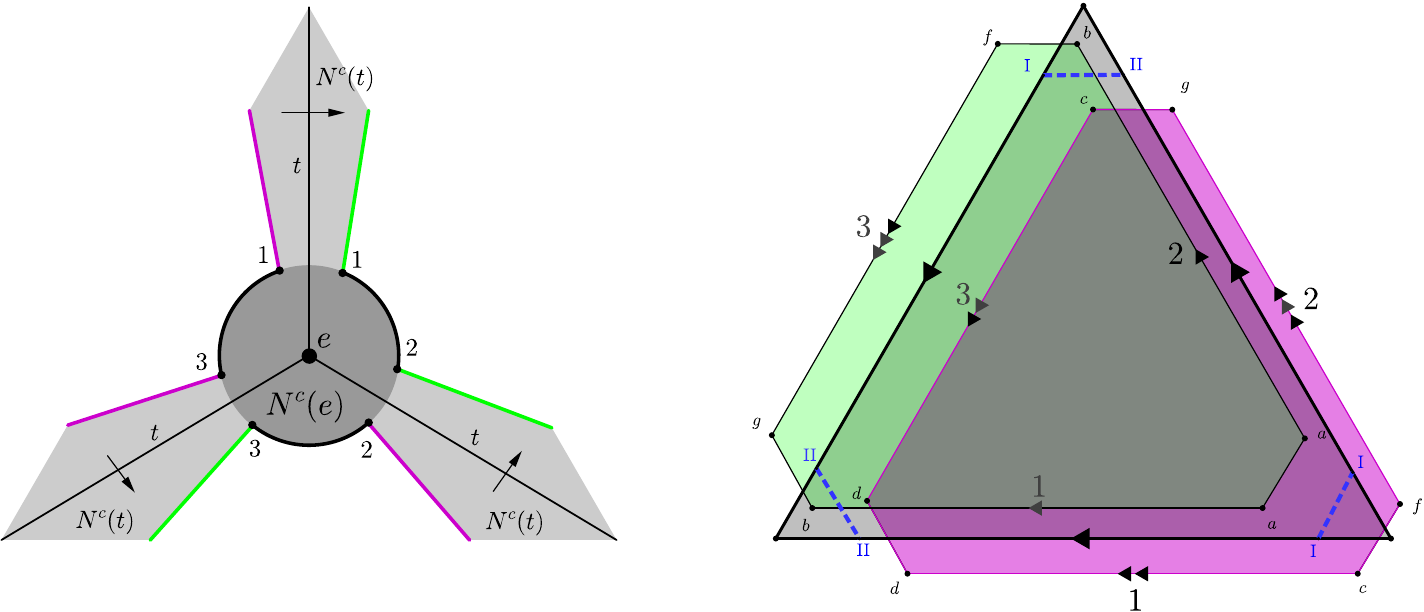}
 \caption{Left: Cross-section through the neighbourhood $N^c (t)$ of an ideal dunce hat $t$ near its boundary edge $e.$ Labels coincide with the ones on the right. Right: Truncated boundary $N^c(t)$ of regular neighbourhood of the dunce hat $t$ (drawn in the centre in black). The boundary surface is connected, of Euler characteristic $-1$ with three boundary components (one containing four edges and two containing a single edge, as indicated by the edge labels) and thus a 3--punctured sphere. The dotted lines with vertex labels represent the intersection of $t$ with the boundary torus $T_0,$ two disjoint loops and a connecting arc.}
 \label{fig:duncehat}
 \end{center}
\end{figure}

The two hexagons again form a properly embedded 3--punctured sphere $\mathcal{S} \subset \core.$ The dunce hat $t$ itself meets $\partial \core$ in two nodes and three normal arcs: two loops, one at each node, and one arc connecting the two nodes -- also known as the {\em barbell graph}. Up to the action of the mapping class group, there are four possibilities of embedding the barbell graph into a torus boundary component $T_0 \subset \partial \core,$ producing three boundary components in a regular neighbourhood of the graph: One where both loops are essential in $T_0$ and thus necessarily parallel; one where one loop is inessential and the other one is an arbitrary essential curve in $T_0;$ and two where both loops are inessential. See Figure~\ref{fig:embeddings_2} for a picture of all four cases. Using the same pasting argument as in Case 1 we can conclude that 
$\mathcal{S}$ is boundary parallel.
Again, uniting $\mathcal{S}$ with the complement of $N(t \cap T_0)$ produces a torus boundary parallel to $T_0$ and disjoint from $e.$ Thus, Theorem~\ref{lem:bdEfficient} produces a contradiction.
\end{proof}

\begin{figure}[htb]
 \begin{center}
 \includegraphics[width=\textwidth]{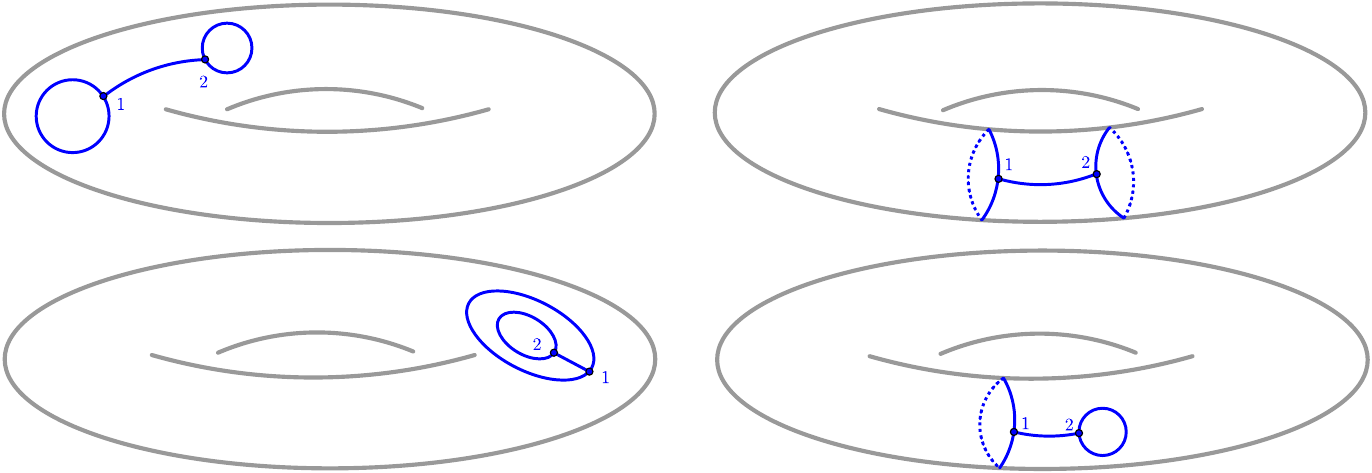}
 \caption{Four embeddings of the intersection of the dunce hat $t$ with torus boundary component $T_0.$ Top left: two inessential loops, not nested. Bottom left: two inessential loops, nested. Top right: two parallel essential loops. Bottom right: one essential loop and one inessential loop.}
 \label{fig:embeddings_2}
 \end{center}
\end{figure}

\begin{remark}
\label{rem:fmp}
The preconditions of Theorem~\ref{thm:no3fold} are necessary: the minimal triangulation of the Gieseking manifold has dunce faces, and there 
exists a minimal 2-tetrahedron triangulation of a 3--manifold with two Klein-bottle cusps, in which one of the four faces is a 3--fold face.

Moreover, Frigerio, Martelli and Petronio describe classes of orientable hyperbolic manifolds $\mathcal{M}_{g,k}$ with $k$ torus cusps and one end of type an orientable surface of genus $g$ \cite{Frigerio-Dehn-2003}. They show, that these manifolds admit minimal $(g+k)$--tetrahedra ideal triangulations if and only if $g>k,$ or $g=k$ and $g$ even, see \cite[Proposition 1.4]{Frigerio-Dehn-2003}. Such triangulations must necessarily have $3k$ cone faces, and $(2g-k)$ dunce- or 3--fold faces, see \cite[Lemma 2.1]{Frigerio-Dehn-2003}. Hence, for $g$ large enough, there exist minimal ideal triangulations of orientable hyperbolic 3--manifolds with all but an arbitrarily small portion of faces being dunce- or 3--fold-faces.
\end{remark}


\subsection{Low degree edges in minimal ideal triangulations}

We now show that there are no edges of degree one or two in a minimal triangulation of a cusped hyperbolic 3--manifold, and that edges of degree three must be contained in layered $\partial$--punctured solid tori.

\begin{lemma} 
  \label{thm:degone}
  If $\tri$ is a minimal ideal triangulation of a cusped hyperbolic 3--manifold $\manifold ,$ then $\tri$ has no edge of degree one.    
\end{lemma}

\begin{proof}  
  Suppose $e_1$ is an edge of $\tri$ of degree 1. Let $\tilde{\Delta}_1$ be the single tetrahedron containing $e_1.$ Then $\tilde{\Delta}_1 = e_1\ast e$ is the join of $e_1$ and the edge $e$ opposite $e_1.$ The edge $e$ of $\tri$ bounds a disc. This contradicts Corollary~\ref{cor-min-bddry-eff}. 
\end{proof} 

\begin{lemma} 
\label{lem:degree two}
  If $\tri$ is a minimal ideal triangulation of a cusped hyperbolic 3--manifold $\manifold ,$ then $\tri$ has no edge of degree two. 
\end{lemma}

\begin{proof}
The \emph{figure-eight} knot complement {\tt m004}\footnote{We include the names of manifolds from the cusped hyperbolic manifold census of SnapPea.} and its sister {\tt m003} are the only cusped hyperbolic 3--manifold having an ideal triangulation with two ideal tetrahedra. Neither of these triangulations has an edge of degree two. Hence, we may assume $\tri$ has at least three tetrahedra and that the preimage of an edge $e$ of degree two consists of two edges $e'$ and $e''$ in two distinct tetrahedra, $\tilde{\Delta}'$ and $\tilde{\Delta}'',$ respectively (note that an edge of degree two cannot be contained in a single tetrahedron). 

We fix the following notation.  The vertices of $\tilde{\Delta}'$ are denoted by $A',B',C',$ and $D'$ with $e' = A'D'.$  The vertices of $\tilde{\Delta}''$ are denoted by $A'',B'',C'',$ and $D''$ with $e'' = A''D''.$  The face identifications are $(A'B'D')\leftrightarrow (A''B''D'')$ and $(A'C'D')\leftrightarrow (A''C''D''),$ with $(A'D') = (A''D'') = e.$ 

\begin{figure}[htb]
 \begin{center}
 \includegraphics[height=4cm]{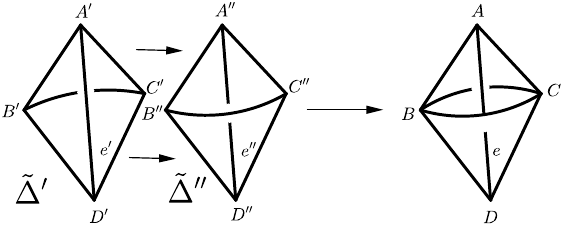}
 \caption{Edge of degree two. \label{f-edge-degree-2}}
 \end{center}
\end{figure}

Typically, in such a case of an edge of degree two, one can discard the two tetrahedra $\tilde{\Delta}'$ and $\tilde{\Delta}''$ and make new identifications $(A'B'C')\leftrightarrow (A''B''C'')$ and $(B'C'D')\leftrightarrow (B''C''D'').$

If this is possible, then there is a triangulation of $\manifold$ having fewer tetrahedra than $\tri,$ contradicting that $\tri$ is minimal. It follows that under the assumption of $\tri$ minimal and $e$ an edge of degree two, there must be an obstruction to such a re-identification, that is, some of the faces to be identified must already be identical. This can happen on the level of edges (Obstruction 1 below), or on the level of triangles (Obstruction 2 below). We consider both obstructions and show that each leads to a contradiction, establishing that there are no edges of degree two.

\noindent{\bf Obstruction 1.} The edges $B'C'$ and $B''C''$ are identified, preventing the collapse, see Figure~\ref{fig:obstruction1}. In this situation we have two subcases:

\begin{enumerate}
  \item Edges are identified with opposite orientation (i.e., $B'C'$ is identified with $C''B''$). If this is the case, then in the truncated 3--manifold $\core$ of $\manifold$ with an open neighbourhood of the vertices removed, we have a properly embedded M\"{o}bius band $\mathcal{N} \subset \core.$ If $\partial \mathcal{N}$ is an inessential curve on the vertex linking torus $T_0 \subset \partial \core$, $\manifold$ admits an embedded projective plane and we have a connected summand of $\R P^3$ in $\manifold,$ contradiction to $\manifold$ being hyperbolic. Hence, assume that $\partial \mathcal{N}$ is an essential curve on $T_0$. From this situation we can construct a Seifert fibred space properly embedded in $\manifold$ which is impossible since our manifold is hyperbolic. To see this, note that the boundary of a small regular neighbourhood of $T_0 \cup  \mathcal{N}$ is a torus. Since our manifold is hyperbolic this torus must be inessential, and thus it must bound a solid torus outside $T_0 \cup  \mathcal{N}.$ Hence the Seifert structure on $T_0 \cup  \mathcal{N}$ extends to a Seifert fibering over this solid torus with at most two exceptional fibres, one at the centre line of the M\"{o}bius band $\mathcal{N}.$ 
  \item Edges are identified with coinciding orientation (i.e., $B'C'$ is identified with $B''C''$). In this case we have a properly embedded annulus $\mathcal{A} \subset  \core.$ Note that, since $\manifold$ is hyperbolic, $\mathcal{A}$ must be either compressible or boundary-parallel. Hence, if one boundary component of $\mathcal{A}$ is an essential curve, the other one is, too, and both are parallel curves on the same torus boundary component $T_0 \subset \partial \core .$ In this case the annulus $\mathcal{A}$ runs parallel to $T_0$ with either $AB$ and $AC$ or $BD$ and $CD$ contained in the region between $\mathcal{A}$ and $T_0$. Now consider a torus running parallel to $T_0$ and bounding a collar neighbourhood of $T_0$ in $\core$ containing $\mathcal{A}$. Normalising this torus cannot result in a surface normally isotopic to $T_0$ because $\mathcal{A}$ acts as a barrier. Accordingly, the triangulation is not boundary efficient. This gives a contradiction to the triangulation being minimal due to Theorem~\ref{lem:bdEfficient}. 

If the boundary components are inessential, colour the bounding discs in the respective torus boundary components. More precisely, colour the disc near vertex $B$ {\em purple}, and the one near vertex $C$ {\em green}. W.l.o.g, edges $AB$ and $AC$ (truncated to run inside $\core$) run inside the solid torus bounded by $\mathcal{A}.$ By construction, the endpoint near $B$ is purple and the endpoint near $C$ is green. Moreover, the endpoints near $A$ must either both be purple or both be green (note that $\mathcal{A}$ must be compressible). It follows that one of the edges $AB$ and $AC$ has endpoints with equal colours. It thus follows that this edge must be in the neighbourhood of the respective torus boundary components of $\core,$ see Figure~\ref{fig:obstruction1_2}. Again, we use Theorem~\ref{lem:bdEfficient} to produce a contradiction to the fact that the triangulation is minimal.
\end{enumerate}

It follows that Obstruction 1. cannot occur.

\begin{figure}[htb]
 \begin{center}
 \includegraphics[height=4cm]{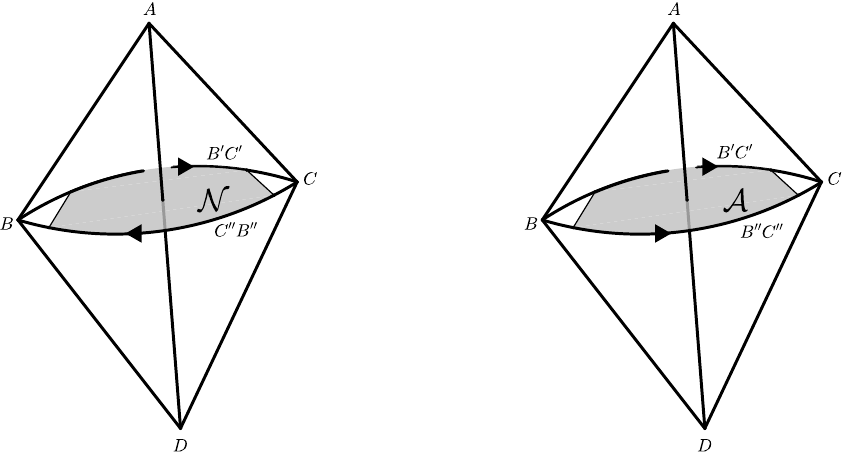}
 \caption{Obstruction 1 (1). Edge $B'C'$ identified with edge $C''B''$ forming a punctured $\R P^2,$ or M\"obius strip $\mathcal{N}$ (left). Obstruction 1 (2). Edge $B'C'$ identified with edge $B''C''$ forming an annulus $\mathcal{A}$ (right).}
 \label{fig:obstruction1}
 \end{center}
\end{figure}

\begin{figure}[htb]
 \begin{center}
 \includegraphics[width=.6\textwidth]{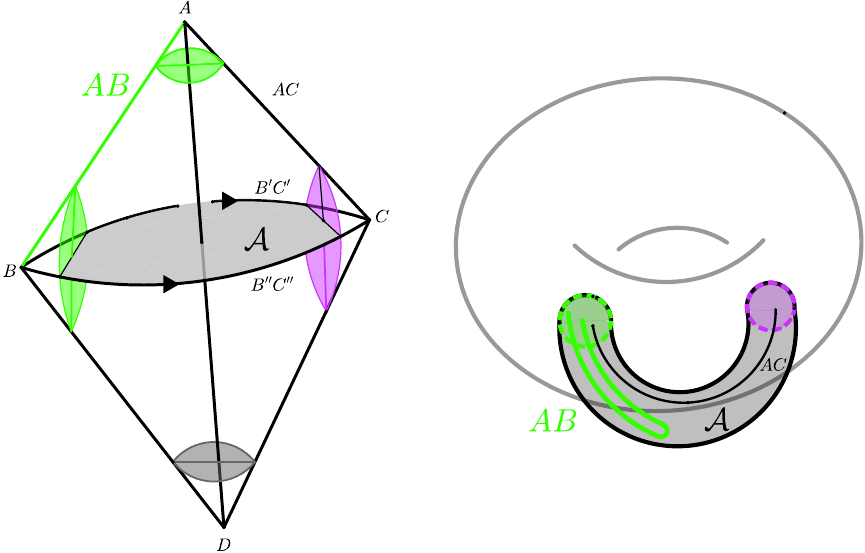}
 \caption{Obstruction 1 (2). Both boundary components of $\mathcal{A}$ are inessential $\partial \core$.}
 \label{fig:obstruction1_2}
 \end{center}
\end{figure}

\noindent{\bf Obstruction 2.} Faces from the front are identified with faces of the back (cf. Figure~\ref{f-edge-degree-2}). As a result of such a face identification, the two edges $B'C'$ and $B''C''$ may become identified. If this the case we fall back onto the cases dealt with in Obstruction~1. Again, we have two subcases:

\begin{enumerate}
  \item Face $(A'B'C')$ is identified with the face $(B''C''D'')$ or $(A''B''C'')$ is identified with the face $(B'C'D').$ These situations are symmetric. As explained above we may assume that $B'C'$ and $B''C''$ are not identified. Suppose we flatten the two tetrahedra so that $(A'B'C')$ is identified with $(A''B''C'')$ and $(B'C'D')$ is identified with $(B''C''D'').$ Since the interiors of $B'C'$ and $B''C''$ were initially disjoint, the preimage of a point under this operation is either a point or an interval. In particular this flattening operation is cell-like and hence preserves the homeomorphism type of the manifold $\manifold.$ This gives a new triangulation of $\manifold$ with fewer tetrahedra which contradicts our assumption that the triangulation $\tri$ is minimal. 

  \item The face $(A'B'C')$ is identified with the face $(A''B''C'')$ or the face $(B'C'D')$ is identified with the face $(B''C''D'').$ Again, these situations are symmetric. However, if face $(A'B'C')$ is glued to face $(A''B''C'')$ without a twist, $B'C'$ is glued to $B''C''$ and we are done (alternatively, if $(A'B'C')\leftrightarrow (A''B''C''),$  then the vertex at $A' =A''$ is a manifold point and $\tri$ is not an ideal triangulation). If the two faces are identified with a twist, all edges of the two triangles become identified. In particular $B'C'$ and $B''C''$ are identified and, again, we are done. 
\end{enumerate}

Since we have that $\tri$ contains at least three tetrahedra, and that all possibilities of having an edge of degree two lead to a contradiction, there are no edges of degree two. 
\end{proof}

\begin{lemma}
\label{lem:lst no folds}
Let $\tri$ be a minimal ideal triangulation of a cusped hyperbolic 3--manifold $\manifold.$ Suppose $\lst^{\star}$ is a layered $\partial$--punctured solid torus and there is a combinatorial map $\varphi\co\lst^{\star}\to \manifold.$ Then $\varphi$ is an embedding.
\end{lemma}

\begin{proof}
We need to show that no identifications of the boundary faces of $\lst^{\star}$ are possible in $\tri.$  
If the two triangles of $\partial \lst^{\star}$ are identified, then $\manifold$ is a punctured lens space, contrary to the assumption that $\manifold$ is hyperbolic.

Hence, identifications can occur at most at the edges of $\partial \lst^{\star}.$ But it follows from Theorem~\ref{thm:no3fold} that not all three edges of $\partial \lst^{\star}$ can be identified, and thus we only need to consider the case that two edges of $\partial \lst^{\star}$ are glued together. 

A priori there are a total of six choices of how two boundary edges of $\partial \lst^{\star}$ can be identified (three pairs of edges and two possible orientations each). Ignoring the third edge, all of them result in a quadrilateral with all of its edges identified in the pattern shown in Figure~\ref{fig:lstbdrglued} in the centre and the third boundary edge is one of he two diagonals of the quadrilateral. Denote the unique boundary edge of the quadrilateral by $e$.

We now pass to a small regular neighbourhood $N = \nhd (\lst^{\star} \cap \core)$ of $\lst^{\star}$ in the compact core $\core$ of $\manifold$. By construction, its boundary $\partial N$ must be a surface with boundary properly embedded in $\core$. We can think of $\partial N$ as the surface consisting of the octagon lying in the interior of the quadrilateral of $\partial \lst^{\star}$, and two quadrilaterals from the annulus that is the regular neighbourhood of the ideal edge $e$ in $\core , $ see Figure~\ref{fig:lstbdrglued} on the right. The surface is obtained by gluing the dotted edges of the octagon to the dotted edges of the quadrilateral. The solid edges then denote the boundary of $\partial N$.

\begin{figure}[htb]
 \begin{center}
 \includegraphics[width=\textwidth]{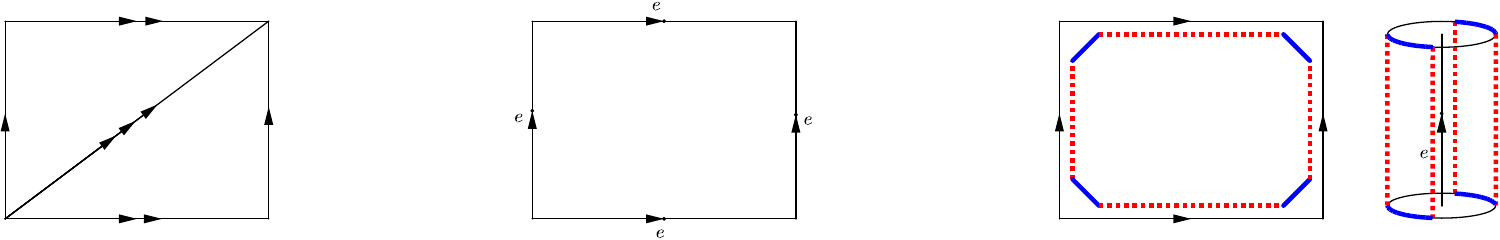}
 \caption{Left: Boundary of $\lst^{\star}.$ Centre: Boundary of $\lst^{\star}$ after disregarding the third boundary edge with two boundary edges identified. Right: Octagon and two quadrilaterals making up the boundary $\partial N$ of a regular neighbourhood of $\lst^{\star}$ in $\core$. Dotted edges are identified, solid edges denote the boundary of $\partial N$.}
 \label{fig:lstbdrglued}
 \end{center}
\end{figure}

Up to symmetry, there are three ways two perform these gluings (notice that the red edges of the octagon can be paired with the red edges of the two quadrilaterals arbitrarily, but the orientation of the gluings is fixed). All of them are shown in Figure~\ref{fig:bdrs}. The first set of gluings results in a once-punctured torus. This type of identification implies that the only identifications of $e$ are the ones from inside $\lst^{\star}$, a contradiction to the assumption that a pair of boundary edges of $\lst^{\star}$ are identified in $\tri$. The second results in $\partial N$ a once-punctured Klein bottle, and the third results in $\partial N$ a thrice-punctured sphere. 

\begin{figure}[htb]
 \begin{center}
 \includegraphics[height=4cm]{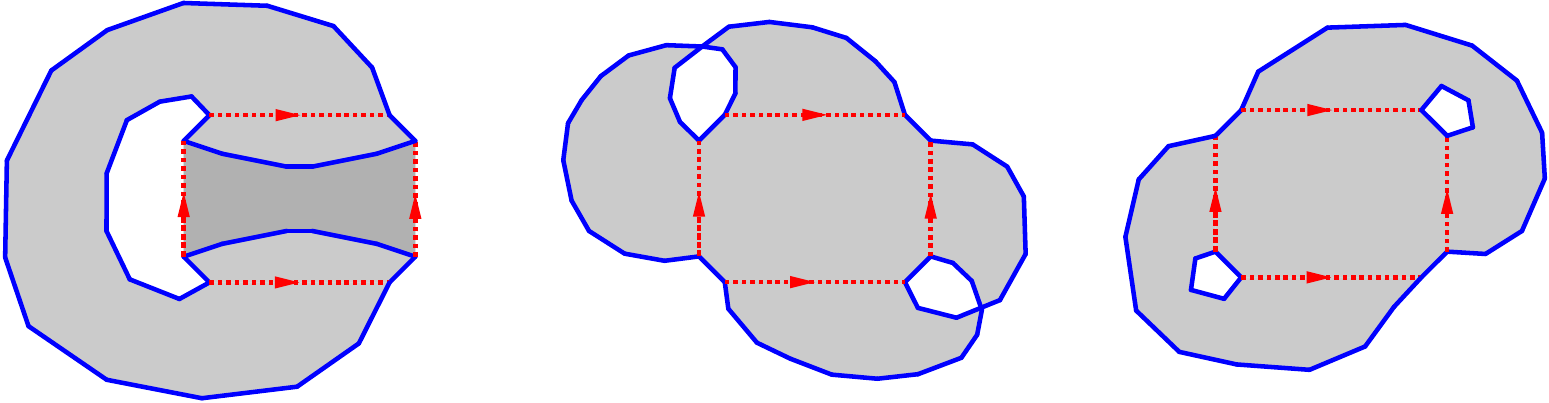}
 \caption{Three possible ways to obtain $\partial N$. Left: a once-punctured solid torus. Centre: a once-punctured Klein bottle. Right: a thrice-punctured sphere.}
 \label{fig:bdrs}
 \end{center}
\end{figure}

Since $\partial N$ is separating it cannot be a once-punctured Klein bottle. Hence, it is a separating thrice-punctured sphere with all boundary components on the same boundary component $T_0$ of $\core$. A parity argument shows that the number of essential boundary components of $\partial N$ must be even. Moreover, note that, by construction, $T_0 \cap N$ is the neighbourhood of a pinched disc in $T_0$ and thus connected. 

It follows that $T_0 \cap N$ must fall into one of two categories.

\begin{enumerate}
  \item Two boundary components of $\partial N$ are essential and hence $T_0 \cap N$ must be equal to an annulus with a disc removed. Since $N$ is separating, this disc must be outside $N$ and thus disjoint from $e$. Hence, pasting this disc into $T_0 \cap N$ yields a necessarily boundary-parallel annulus. 
  \item No boundary component of $\partial N$ is essential and hence $T_0 \cap N$ must be equal to a subsurface of $T_0$ with two discs removed. Again, these discs must be outside $N$ and thus disjoint from $e$. Pasting the discs into $T_0 \cap N$ results in a necessarily boundary-parallel disc.
\end{enumerate}

Now, as in earlier arguments, consider a torus running parallel to $T_0$ and bounding a collar neighbourhood of $T_0$ in $\core$ containing the annulus or disc. Normalising this torus cannot result in a surface normally isotopic to $T_0$ because it must be disjoint from $e$. Accordingly, the triangulation is not boundary efficient. This gives a contradiction to the triangulation being minimal due to Theorem~\ref{lem:bdEfficient} and thus $\varphi\co\lst^{\star}\to \manifold$ must be an embedding.
\end{proof}

\begin{lemma}
  \label{lem:degreethree}
  Let $\tri$ be a minimal ideal triangulation of a cusped hyperbolic 3--manifold $\manifold$ and $e$ be an ideal edge of degree three. Then $e$ is contained in an embedded two tetrahedron subcomplex of $\tri$ combinatorially equivalent to the layered $\partial$--punctured solid torus $\lst^\star(1, 3, 4).$ 
In particular, $e$ is contained in a maximal layered $\partial$--punctured solid torus $\lst^\star$, twice in its core tetrahedron and once in the tetrahedron adjacent to the core, and no other edge in $\lst^\star$ is of degree three.

Moreover, $\tri$ contains at least three ideal tetrahedra and each ideal edge of degree three is contained in its own subcomplex of this type.
\end{lemma}

\begin{proof}
We may and will work with the pseudo-manifold $\widehat{\manifold}.$ Let $e \in \tri$ be an edge of degree $3.$ If $e$ is incident with three distinct tetrahedra, then a $3$-2--move produces a triangulation with fewer tetrahedra, contradicting minimality of the triangulation. If $e$ is incident with exactly one tetrahedron, it cannot be of degree $3.$ Hence, $e$ is incident with exactly two tetrahedra. A brute-force enumeration shows that there exists only one 2--tetrahedra complex with an edge of degree $3$ in its interior -- the layered solid torus $\lst(1,3,4)$ (see case $|\tilde{\Delta}_e|=2$ in the proof of \cite[Proposition 9]{Jaco-minimal-2009} for details). Moreover, $\lst(1,3,4)$ has exactly one such degree $3$ edge in its interior. This edge occurs twice in its core tetrahedron and once in the other tetrahedron (i.e., the tetrahedron adjacent to the core tetrahedron).

It follows from Lemma~\ref{lem:lst no folds} that no identifications between the faces of this layered solid torus are possible, hence the subcomplex is embedded, contained in a maximal layered $\partial$--punctured solid torus and, in particular, the triangulation contains at least three tetrahedra.
\end{proof}

We remark that an analysis of edges of degree four and five can be carried out following \cite{Jaco-minimal-2009}, but this is not needed for the main applications of this paper and hence omitted.


\subsection{Layered punctured solid tori in ideal triangulations}

The argument for closed 3--manifolds in \cite{Jaco-Z2-2013} hinged on an understanding of the intersections of \emph{maximal} layered solid tori. We now carry such an analysis out in the case of cusped hyperbolic 3--manifolds. This requires a different approach due to the fact that the combinatorics of minimal ideal triangulations is less restricted than the combinatorics of minimal and 0--efficient (material) triangulations.
 
As in \cite{Jaco-Z2-2013} we focus on subcomplexes in triangulations that are maximal layered $\partial$--punctured solid tori (see Section~\ref{sec:lst} for a brief introduction).

\begin{lemma}
  \label{lem:intersection}
  Let $\tri$ be a minimal ideal triangulation of a cusped hyperbolic 3--manifold $\manifold.$ Then 
  the intersection of any two maximal layered $\partial$--punctured solid tori is either empty, an ideal vertex, or an ideal edge of~$\tri.$   
\end{lemma}

\begin{proof}
Let $\lst_a, \lst_b \subset \tri$ be two maximal $\partial$--layered solid tori. If their boundaries $\partial \lst_a$ and $\partial \lst_b$ meet in two faces, $\tri$ must be homeomorphic to a punctured lens space. A contradiction. If $\lst_a$ and $\lst_b$ share a common tetrahedron, remove the tetrahedra in the intersection from one of them and observe, that they now meet in two faces. Again, a contradiction. Hence, we can assume that two maximal $\partial$--layered solid tori meet along a single face, three edges, or two edges.

\begin{enumerate}
    \item Assume that $\partial \lst_a$ and $\partial \lst_b$ meet along exactly two edges. It follows, that they must be identified along a spine of each of $\partial \lst_a$ and $\partial \lst_b,$ referred to as their {\em common spine}. Let $\mathcal{D}_a$ and $\mathcal{D}_b$ be the discs obtained by taking the interior of $\partial \lst_a$ and $\partial \lst_b$ after truncating their vertex and removing their common spine. $\mathcal{D}_a$ and $\mathcal{D}_b$ can be glued together along sections of the exchange annuli in a neighbourhood of the common spine yielding a 4--punctured sphere $\mathcal{S}$ properly embedded in $\core$ (the truncated compact core of $\manifold$) and disjoint from the two edges of the common spine.

By construction, the four boundary components of $\mathcal{S}$ in one of the torus boundary components $T_0\subset \partial \core$ bound a regular neighbourhood of the intersection of the maximal layered solid tori $\lst_a$ and $\lst_b$ with $T_0$ -- a regular neighbourhood of two discs pinched along four points (the four points being the intersection of the common spine with $T_0$). Since this regular neighbourhood is connected, of Euler characteristic $-2,$ and with four boundary components, it must be a 4--punctured sphere contained in $T_0.$ There are two cases: either all components of $\partial\mathcal{S}$ are trivial in $T_0$ (that is, $\mathcal{S}$ equals a disc of $T_0$ with three points removed), or two components of $\partial\mathcal{S}$ are non-trivial (that is, $\mathcal{S}$ equals an annulus of $T_0$ with two points removed).

First suppose $(\lst_a \cup \lst_b) \cap T_0$ is contained in a disc on $T_0$ and hence three of its boundary components bound discs in the complement of the intersection. Pasting these three discs of $T_0$ into $\mathcal{S}$ and pushing them slightly off $T_0$ yields a disc properly embedded in $\core.$ This disc can be extended to a torus parallel to boundary component $T_0$ of $\core$. Moreover, this torus is disjoint from the two edges of $\tri$ of the common spine, which both are incident to $T_0$. It follows that a normalised version of this torus cannot be equal to the vertex linking torus $T_0$ and thus $\tri$ cannot be boundary efficient. A contradiction to Theorem~\ref{lem:bdEfficient}. 

Hence suppose $(\lst_a \cup \lst_b) \cap T_0$ is contained in an annulus on $T_0$ and hence two of its boundary components bound discs in the complement of the intersection. Pasting these two discs of $T_0$ into $\mathcal{S}$ and pushing them slightly off $T_0$ yields an annulus properly embedded in $\core$ and disjoint from the two edges of $\tri$ of the common spine (which are both incident to $T_0$). This annulus together with the annulus in $T_0$ complementary to $(\lst_a \cup \lst_b) \cap T_0$ extends to a boundary parallel torus since otherwise $\manifold$ would contain a non-trivial Seifert fibred space.
Again, this boundary parallel torus leads to a contradiction to the assumption that $\tri$ is boundary efficient (cf. Theorem~\ref{lem:bdEfficient}). 

    \item Assume that $\partial \lst_a$ and $\partial \lst_b$ meet along exactly three edges. Following an analogous procedure as in the previous case, we arrive at two 3--punctured spheres $\mathcal{S}_1$ and $\mathcal{S}_2$ (instead of one 4--punctured sphere $\mathcal{S}$), both properly embedded in $\core,$ both meeting the same boundary component $T_0 \subset \partial \core.$ Again, it follows that a regular neighbourhood of $(\lst_a \cup \lst_b) \cap T_0$ is connected, of Euler characteristic $-4$ with six boundary components and thus a $6$--punctured sphere. 

Pasting in some of the boundary discs of this $6$--punctured sphere, analogously to the procedure carried out in the previous case, yields a torus parallel to boundary component $T_0$ and disjoint from three edges of the triangulation which are all incident to $T_0$ -- and we conclude as before by using Theorem~\ref{lem:bdEfficient}.  

\item Assume that $\partial \lst_a$ and $\partial \lst_b$ meet along a single face. Again, we argue as in the previous cases. Here, we end up with one 3--punctured sphere $\mathcal{S}$ (the other one being filled in as a result of the face identification). A regular neighbourhood of $(\lst_a \cup \lst_b) \cap T_0$ now is connected, of Euler characteristic $-1,$ has three boundary components, and thus is contained in a disc or an annulus on $T_0.$  We again obtain a disc or an annulus that can be extended to a torus running parallel to $T_0$ which is disjoint of three edges of the triangulation incident to $T_0$ and we use Theorem~\ref{lem:bdEfficient} once again to finish the proof. 
\end{enumerate}
This completes the proof of the lemma.
\end{proof}


\section{Normal surfaces representing $\Z_2$--homology classes}
\label{sec:normsurfs}

Normal surfaces in 1--vertex triangulations of closed 3--manifolds that are dual to non-trivial first cohomology classes were used in \cite{Jaco-minimal-2009, Jaco-Z2-2013} in order to obtain lower bounds on the complexity of 3--manifolds. This section develops a similar theory for ideal triangulations. In this section, $\closure$ is a compact, irreducible 3--manifold with non-empty boundary a finite union of tori, and $\manifold$ is its interior. 

We consider surfaces representing elements of $H_2(\manifold; \Z_2).$ We wish to emphasise that we do not work with $H_2(\closure, \partial \closure; \Z_2).$ 
We recall the following definitions from the introduction.
If $S$ is connected, let $\chi_{-}(S) = \max \{ 0,-\chi(S)\},$ and otherwise let
\[\chi_{-}(S) = \sum_{S_i\subset S} \max \{ 0,-\chi(S_i)\},\]
where the sum is taken over all connected components of $S.$ Note that $S_i$ is not necessarily orientable. Define:
\[|| \ c \ || = \min \{ \chi_{-}(S) \mid [S]= c\}.\]
Then the surface $S$ representing the class $c \in H_2(\manifold;\Z_2)$ is \emph{$\Z_2$--taut} if no component of $S$ is a sphere, real projective plane, Klein bottle or torus and $\chi(S) = -|| \ c \ ||.$ As in \cite{Thurston-norm-1986}, one observes that every component of a $\Z_2$--taut surface is non-separating and geometrically incompressible.

\subsection{Rank-1 colouring of edges and the canonical surface}
\label{subsec:colouring}

We have $\manifold = \widehat{\manifold}\setminus \widehat{\manifold}^{(0)},$ and thus Lefschetz duality gives a natural isomorphism 
\[D_\manifold \co H_2(\manifold; \Z_2) 
\cong H^1(\widehat{\manifold}, \widehat{\manifold}^{(0)}; \Z_2).\]
An element $\varphi \in H^1(\widehat{\manifold}, \widehat{\manifold}^{(0)}; \Z_2)$ can be viewed as an assignment of values $0$ or $1$ to each edge in the triangulation such that in each triangle the sum of all values at its edges equals zero modulo two. We say that the edge $e$ is \emph{$\varphi$--even} if $\varphi(e) = 0$ and \emph{$\varphi$--odd} if $\varphi(e) = 1.$
It follows that a tetrahedron falls into one of the following categories, which are illustrated in Figure~\ref{fig:Z2-homology class}:

\begin{itemize}
\item[] Type $\Tq$: A pair of opposite edges are $\varphi$--even, all others are $\varphi$--odd.
\item[] Type $\Tt$: The three edges incident to a vertex are $\varphi$--odd, all others are $\varphi$--even.
\item[] Type $\Tee$: All edges are $\varphi$--even.
\end{itemize}

\begin{figure}[htb]
  \centerline{\includegraphics[height=2.4cm]{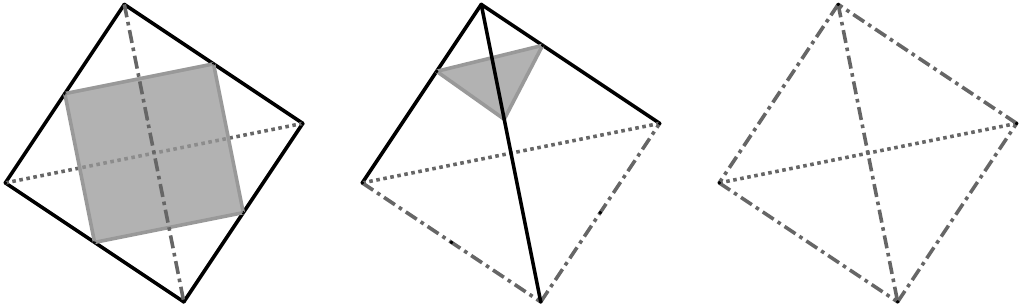}}
  \caption{Tetrahedra of type $\Tq,$ $\Tt$ and $\Tee.$ Grey and dashed edges are $\varphi$--even. \label{fig:Z2-homology class}}
\end{figure}

If $\varphi$ is non-trivial, one obtains a unique normal surface, $S_\varphi = S_\varphi(\tri),$ with respect to $\tri$ by introducing a single vertex on each $\varphi$--odd edge. This surface is disjoint from the tetrahedra of type $\Tee;$ it meets each tetrahedron of type $\Tt$ in a single triangle meeting all $\varphi$--odd edges; and each tetrahedron of type $\Tq$ in a single quadrilateral dual to the $\varphi$--even edges. Moreover, $S_\varphi$ is Lefschetz dual to $\varphi$ and is termed the \emph{canonical (normal) surface dual to $\varphi.$} 
Given $0 \neq c \in H_2(\manifold; \Z_2)$ we also write $S_c = S_{D_\manifold(c)}$ and call $S_c$ the \emph{canonical (normal) representative} of $c.$

\begin{lemma}\label{lem:chi bounds norm for canonical}
Let $\manifold$ be a cusped hyperbolic 3--manifold with ideal triangulation $\tri$ and $0 \neq c \in H_2(\manifold; \Z_2).$ Then $|| \ c \ || \le \chi_{-} (S_c) = - \chi (S_c).$ 
\end{lemma}

\begin{proof}
It follows from the construction and definitions that $[S_c]=c$ and we need to prove the equality in the statement. For this, it suffices to show that no component of $S_\varphi$ is a sphere or a projective plane. 
Since $\manifold$ is irreducible and not $\R P^3,$ no component can be a projective plane. Since each component of $S_\varphi$ meets some edge, and it meets each edge in at most one point, no component can be a sphere.
\end{proof}


\subsection{Rank-2 colouring of edges}
\label{subsec:rank-2 colouring}

Given the subgroup $H = \langle \varphi_1, \varphi_2 \rangle \cong \Z_2 \oplus \Z_2$ of $H^1(\widehat{\manifold}, \widehat{\manifold}^{(0)}; \Z_2),$ we now introduce a refinement of the above colouring. Since $\varphi_1 + \varphi_2 = \varphi_3,$ there are four types of edges:
\begin{itemize}
\item[] edge $e$ is called \emph{$H$--even} or \emph{0--even} if $\varphi_i[e] =0$ for each $i \in \{ 1,2,3\};$ and
\item[] edge $e$ is called $i$--even if $\varphi_i[e] =0$ for a unique $i \in \{ 1,2,3\}.$
\end{itemize}
Let $\{ i,j,k\} = \{ 1,2,3\}.$ The normal corners of the normal surface $S_{\varphi_i}(\tri)$ are precisely on the $j$--even and $k$--even edges. Edge $e$ is $\varphi_i$--even if it is $i$--even or 0--even. It follows that a face of a tetrahedron either has all of its edges 0--even; or it has two $1$--even (or $2$-even or $3$-even respectively) and one 0--even edge; or it has one 1--even, one 2--even and one 3--even edge. Whence a tetrahedron falls into one of the following categories:

\begin{itemize}
\item[] Type $\Tqtt$: One edge is 0--even, the opposite edge is $i$--even, and one vertex of the latter is incident with two $j$--even edges, and the other with two $k$--even edges, where $\{i,j,k\} = \{ 1,2,3\}.$ (There are six distinct colourings of such a tetrahedron, in the following referred to as {\em sub-types}.)
\item[] Type $\Tqq$: A pair of opposite edges are 0--even, all others are $i$--even for a unique $i \in \{ 1,2,3\}.$ (There are hence three distinct sub-types.)
\item[] Type $\Ttt$: The three edges incident to a vertex are $i$--even for a fixed $i \in \{ 1,2,3\},$ and all others are 0--even. (There are hence three distinct sub-types.)
\item[] Type $\Tee$: All edges are 0--even.
\item[] Type $\Tqqq$: Each vertex is incident to an $i$--even edge for each $i \in \{ 1,2,3\}.$ (In particular, no edge is 0--even, opposite edges are of the same type, and there are two distinct sub-types of tetrahedra.)
\end{itemize}
For each type, one sub-type is shown in Figure \ref{fig:multiZ2-homology class}; the dashed grey edges correspond to 0--even edges and the normal discs in $S_{\varphi_i}$ have the same colour as the $i$--even edges. The remaining sub-types are obtained by permuting the colours.

\begin{figure}[htb]
  \centerline{\includegraphics[height=2.4cm]{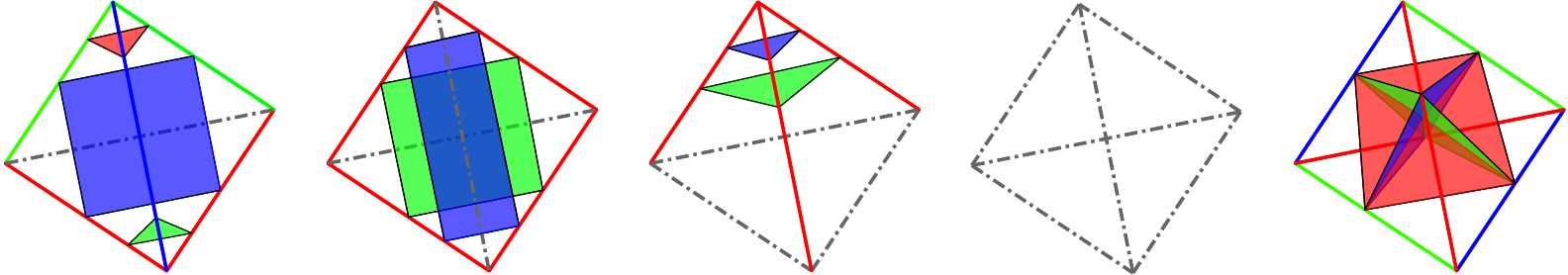}}
  \caption{Tetrahedra of type $\Tqtt,$ $\Tqq,$ $\Ttt,$ $\Tee$ and $\Tqqq.$ Dashed grey edges are 0--even. \label{fig:multiZ2-homology class}}
\end{figure}


\subsection{Combinatorial bounds for triangulations}

The set-up and notation of the previous subsection is continued. Let
\begin{align*}
\nqtt(\tri)&= \text{ number of tetrahedra of type } \Tqtt, \\
\nqq(\tri)&= \text{ number of tetrahedra of type }\Tqq, \\
\ntt(\tri)&= \text{ number of tetrahedra of type }\Ttt, \\
\nee(\tri)&= \text{ number of tetrahedra of type }\Tee, \\
\nqqq(\tri)&= \text{ number of tetrahedra of type }\Tqqq, \\
\even(\tri)&= \text{ number of 0--even edges, }\\
\tilde{\even}(\tri)&= \text{ number of preimages of 0--even edges in } \widetilde{\Delta}.
\end{align*}
The number of tetrahedra in $\tri$ is 
$$|\tri| = \nqtt(\tri) + \nqq(\tri) + \ntt(\tri) + \nee(\tri) + \nqqq(\tri).$$ 
For the remainder of this subsection, we write $\nqtt = \nqtt(\tri),$ etc. 

Let $\hat{K}$ be the complex in $\tri$ spanned by all 0--even edges (i.e., the complex of all edges, faces and tetrahedra of $\tri$ containing only 0--even edges) and let $K$ be the ideal complex obtained from $\hat{K}$ by removing its vertices. 

Let $N$ be a small regular neighbourhood of $K.$ Then $\partial N$ is a normal surface; it meets each tetrahedron in the same number and types of normal discs as $S_{\varphi_1} \cup S_{\varphi_2}\cup S_{\varphi_3}$ except for the tetrahedra of type $\Tqqq,$ which it meets in four distinct normal triangle types instead of three distinct normal quadrilateral types. The normal coordinate of $\partial N$ is thus obtained by taking the sum of the normal coordinates of $S_{\varphi_1},$ $S_{\varphi_2},$ and $S_{\varphi_3},$ and adding to this the \emph{tetrahedral solution} of each tetrahedron of type $\Tqqq.$ (The tetrahedral solution is obtained by adding all triangle coordinates of a tetrahedron and subtracting all quadrilateral coordinates; see, for instance, \cite[Section]{Kang-ideal-2004}, or \cite[Section 2.8]{Tillmann-normal-2008}. Each tetrahedral solution has both positive and negative coordinates.) Hence
$$
\chi(S_{\varphi_1}) +\chi (S_{\varphi_2}) + \chi (S_{\varphi_3}) + \nqqq = \chi (\partial N) = 2 \chi(N) = 2 \chi(K) = - 2 \even +\ntt + 2\nee,$$
where the rightmost equality is a consequence of the following computation: By definition, the Euler characteristic of $K$ equals minus the number of 0--even edges plus the number of triangles with only 0--even edges minus the number of tetrahedra of type $\Tee$ (note that $K$ is ideal and hence has no vertices). Since every triangle occurs in exactly two tetrahedra, every tetrahedron of type $\Ttt$ accounts for one half of a triangle with 0--even edges and every tetrahedron of type $\Tee$ accounts for two triangles of 0--even edges. Hence, in total we have 
$$
\chi (K) = - \even + \frac12 \ntt + 2 \nee - \nee = -\even + \frac12 \ntt + \nee .
$$

In particular, 
$\chi(S_{\varphi_1}) +\chi (S_{\varphi_2}) + \chi (S_{\varphi_3}) + \nqqq$ is even.
Rearranging the above equality gives
\begin{equation}\label{eq:e}
\ntt + 2\nee -\nqqq = 2 \even  + \chi(S_{\varphi_1}) +\chi (S_{\varphi_2}) + \chi (S_{\varphi_3}),
\end{equation}

and hence:
\begin{align}\label{inequ for one vert tri gen}
\tilde{\even} 	&= \nqtt+2\nqq+3\ntt+6\nee\nonumber \\
		&=  2|\tri|  -\nqtt + \ntt +4\nee-2\nqqq  \nonumber\\
		&= 2|\tri|  - \nqtt - \ntt+4 \even+ 2 (\chi(S_{\varphi_1}) +\chi (S_{\varphi_2}) + \chi (S_{\varphi_3})).
\end{align}
\begin{lemma}\label{lem:formula for degree 3 edges}
Let $\manifold$ be a cusped hyperbolic 3--manifold of finite volume, and suppose that $\varphi_1, \varphi_2 \in  H^1(\widehat{\manifold}, \widehat{\manifold}^{(0)}; \Z_2)$ are non--trivial classes with $\varphi_1+\varphi_2 = \varphi_3 \neq 0.$ Let $\tri$ be a minimal triangulation, $S_{\varphi_i}$ be the canonical surface dual to $\varphi_i,$ and let $\even_d$ denote the number of 0--even edges of degree $d.$ Then
\begin{equation}\label{inequ for min tri general}
\even_3= \nqtt +\ntt - 2 (|\tri| + \chi(S_{\varphi_1}) +\chi (S_{\varphi_2}) +\chi (S_{\varphi_3})) + \sum_{d=5}^{\infty} (d-4) \even_d.
\end{equation}
\end{lemma}

\begin{proof}
Lemmata~\ref{thm:degone} and \ref{lem:degree two} imply that the smallest degree of an edge in $\tri$ is three. One has $\tilde{\even} = \sum d \even_d$ and $\even = \sum \even_d.$ Inserting this into Equation~(\ref{inequ for one vert tri gen}) yields the desired equality.
\end{proof}


\section{Quadrilateral surfaces and minimal ideal triangulations}
\label{sec:quadsurfs}

In this section we prove the following result, first stated in the introduction. The essence of the proof is a counting argument taking into account even edges, compression discs for the canonical normal representatives, and types of tetrahedra. This is modelled on the blueprint for closed 3--manifolds \cite{Jaco-Z2-2017, Jaco-minimal-2009, Jaco-Z2-2013, Nakamura-complexity-2017}.

\begin{reptheorem}{thm:sumofnorms}
  Let $\tri$ be an ideal triangulation of a cusped hyperbolic 3--manifold $\manifold$ and let $H~\le~H_2 (\manifold,\Z_2)$ be a rank $2$ subgroup. Then 
  $$|\tri| \geq \sum \limits_{0 \neq c \in H} || \ c \ ||. $$
  Moreover, in the case of equality the triangulation is minimal, each canonical normal representative of a non-zero element in $H\le H_2 (\manifold,\Z_2)$ is taut and meets each tetrahedron in a quadrilateral disc, and the number of tetrahedra in the triangulation is even.
\end{reptheorem}

Note that the conclusion that  \emph{each canonical normal representative meets each tetrahedron in a quadrilateral disc} is equivalent to saying that \emph{all tetrahedra are of type $\Tqqq$}.
  
\begin{proof}
It suffices to work with a minimal ideal triangulation. For if we can show any such a triangulation satisfies the desired inequality, then clearly an arbitrary ideal one does also. Moreover, if any ideal triangulation admits equality, then it clearly must be minimal.

Let $\tri'$ be a minimal ideal triangulation of a cusped hyperbolic 3--manifold $\manifold.$ Fix a subgroup $H< H_2 (\manifold,\Z_2)$ of rank $2,$ and let $S_i,$ $i=1,2,3,$ be the canonical normal representatives  of the non-zero classes  $0 \neq \varphi_i \in H.$ Let $\tri$ be the minimal triangulation that
has the smallest number of tetrahedra of type $\Tee$ amongst all minimal triangulations of $\manifold$ that can be obtained from $\tri'$ by a sequence of 4-4-moves.

Following the notation used in Section~\ref{sec:normsurfs}, the union of the normal surfaces $S_i,$ $ i = 1,2,3,$ intersects the tetrahedra of $\tri$ in $\nqtt$ tetrahedra of type $\Tqtt,$ $\nqq$ tetrahedra of type $\Tqq,$ and so on. 

Let $d_i$ be the maximal number of pairwise disjoint compression discs for $S_i$, whose union does not separate any component
of $S_i.$
We have $|| c_i || \leq - \chi (S_i)$ due to Lemma~\ref{lem:chi bounds norm for canonical}. Moreover, in the presence of compression discs, we obtain the stronger bound 
\[ \sum \limits_{i=1}^{3} ||\ c_i\ || \le - \sum \limits_{i=1}^{3} (\chi(S_i) +2d_i)\]
since compressions do not change the homology class and do not introduce components that are spheres or projective planes.
To simplify the expression, and with a view towards the second part of the theorem, we write 
\[\sum \limits_{i=1}^{3} ||\ c_i\ ||  = -D - \sum \chi(S_i),\]
where $D\ge 0.$ Here, we take care of the possibility that the canonical surfaces are incompressible but not norm minimising, and we note that each independent compression disc contributes 2 to $D.$

For the sake of contradiction, assume
\[|\tri| < \sum \limits_{i=1}^{3} ||\ c_i\ ||.\]
Lemma~\ref{lem:formula for degree 3 edges} now implies the strict inequality
\begin{equation}
  \label{eq:main}
  \ntt+\nqtt+2 D+\sum_{d=5}^{\infty} (d-4) \even_d < \even_3.
\end{equation}

Hence, proving the first part of Theorem~\ref{thm:sumofnorms} is equivalent to showing that this is impossible. That is, we need to prove that the left hand side of inequality~(\ref{eq:main}) must be at least as large as the right hand side.

The idea of the proof is to counter-balance each contribution to the right hand side with a contribution to the left hand side of (\ref{eq:main}) of equal or larger size. A contribution to the right hand side only arises from a degree three edge, and we have already shown that each such edge is contained in a layered $\partial$--punctured solid torus subcomplex. Any such complex contains at most one edge of degree three. The organising principles for the counting argument are as follows. We identify sets of degree three edges that can be associated with an $H$--even edge $f$ such that 
$d(f)-4$ is as least as large as the size of the associated set of degree three edges (see Claim 2 and the paragraph directly after the proof of Claim 2). 
In doing so, we may double count some of the degree three edges, but we make sure not to double count any of edges $f$.
For the remaining degree three edges, there are only a small number of special cases, in which we need to also take into account the remaining terms $\ntt+\nqtt+2 D$. 
Here we need to make sure that such contributions to the left hand side are only counted at most once. Moreover, in these special cases the contribution to the left hand side of (\ref{eq:main}) is strictly greater than the contribution to the right hand side. 

The following terminology helps organise the counting argument.
A \emph{deficit in (\ref{eq:main})} is when the contribution to the left is bigger than to the right, a \emph{balance in (\ref{eq:main})} is when they are equal, and otherwise we have a \emph{gain in (\ref{eq:main})}. Thus, for proving the first part of Theorem~\ref{thm:sumofnorms} it is sufficient to disregard all deficits and balances and show that it is impossible to construct a triangulation exhibiting a gain. For the second part of Theorem~\ref{thm:sumofnorms} we also need to keep track of balances.

Recall that the core tetrahedron of any layered $\partial$--punctured solid torus $\lst$ is a standard one-tetrahedron layered $\partial$--punctured solid torus $\lst^\star(1, 2, 3).$ A canonical normal representative of a class of $H_2 (\manifold,\Z_2)$ necessarily intersects $\lst^\star(1, 2, 3)$ in one particular quadrilateral type, or not at all. This can be seen by enumerating all $0/1$-colourings of edges in $\lst^\star(1, 2, 3)$ such that every face of $\lst^\star(1, 2, 3)$ has zero or two edges of colour $1$. Accordingly, the intersection of $\lst^\star(1, 2, 3)$ with all three canonical normal representatives of non-trivial classes of $H$ must result in a tetrahedron of type $\Tqq$ or $\Tee$. 

It follows that either all tetrahedra in $\lst$ are of type $\Tqq$ or of type $\Tee$ (this follows directly from looking at the pattern of $H$--odd edges in the boundary). Accordingly we say that $\lst$ is of type $\Tqq$ or $\Tee.$ A layered $\partial$--punctured solid torus $\lst$ of type $\Tqq$ has one $H$--even boundary edge and the other two boundary edges are $H$--odd.

\noindent {\bf Claim 0}: All edges of degree 3 in $\tri$ are $H$--even.

\noindent {\bf Proof of Claim 0}: Due to Lemma~\ref{lem:degreethree}, every edge $e \in \tri$ of degree three can be associated to a maximal layered $\partial$--punctured solid torus $\lst_e \subset \tri$ containing $e$ in its interior. More precisely, $e$ occurs twice in the core tetrahedron of $\lst_e$ and once in the tetrahedron adjacent to the core. If $\lst_e$ is of type $\Tee,$ then naturally $e$ must be $H$--even. Hence, let $\lst_e$ be of type $\Tqq.$ Since $e$ occurs twice in the core tetrahedron and the other two edges in the core tetrahedron occur in it one and three times respectively, they only way the core tetrahedron of $\lst_e$ can be of type $\Tqq$ is for $e$ to be $H$--even.\qed

A layered $\partial$--punctured solid torus $\lst$ has exactly one boundary edge of degree 1 with respect to $\lst.$ We call the other two boundary edges \emph{non-unital}. 

\noindent {\bf Claim 1}: A maximal layered $\partial$--punctured solid torus $\lst$ of type $\Tee$ containing an $H$--even edge of degree three either has an interior $H$--even edge of degree $\geq 6$, or non-unital boundary edges $g$ and $f$ with $d_{\lst}(g) = 3$ and $d_{\lst}(f) \geq 5$.

\noindent {\bf Proof of Claim 1}: In a layered solid torus of type $\Tee$ every edge is necessarily $H$--even. We can thus ignore the edge-coulourings. 

Every layered $\partial$--punctured solid torus $\lst$ in a minimal triangulation $\tri$ and containing a degree three edge is obtained from a layered $\partial$--punctured solid torus of type $\lst^{\star} (1,3,4)$ by iteratively layering onto one of the non-unital boundary edges. In $\lst^{\star} (1,3,4)$, the non-unital boundary edges $g$ and $f$ have degrees $d_{\lst}(g) = 3$ and $d_{\lst}(f) = 5$. If we layer on $f$ we create an interior edge of degree six and we are done. If we layer on $g$ we create an interior edge of degree four, the formerly unital boundary edge is now of degree three, and the degree of $g$ -- still being in the boundary -- is increased by two. Layering on the unital edge contradicts minimality, since it creates an edge of degree two. Iterating this procedure proves the claim.
\qed

\medskip

\noindent {\bf Claim 2}: Let $\lst \subset \tri$ be a maximal layered $\partial$--punctured solid torus that contains an interior edge $ e \in \tri$ of degree three. Then $\lst$ is incident with an $H$--even edge $f$ such that if $m$ is the number of degree three edges of all maximal layered $\partial$--punctured solid tori incident with $f$, then $d(f)-4 \ge m$ (i.e. contributing a deficit or balance in (\ref{eq:main})), unless the following three conditions are satisfied for $\lst$:
\begin{enumerate}
  \item all interior $H$--even edges distinct from $e$ are of degree four, 
  \item the solid torus is of type $\Tqq$, and
  \item the unique $H$--even boundary edge $f$ is of degree one with respect to $\lst.$ 
\end{enumerate}
We call a maximal layered $\partial$--punctured solid torus that contains an interior edge $e$ of degree three and satisfies the conditions (1)--(3) {\em internally unbalanced}. 
  
\medskip

\noindent {\bf Proof of Claim 2}: 
No interior edge of $\lst$ distinct from $e$ can be of degree less than four (see Lemma~\ref{lem:degreethree}).
If $\lst$ contains an interior $H$--even edge of degree $> 4,$ then we choose $f$ to be this edge. It follows that $m=1$, and thus  $d(f)-4 \ge m$ holds. Hence assume that all interior $H$--even edges distinct from $e$ are of degree four.

Suppose that  $\lst$ is of type $\Tee$. All edges of $\lst$ are $H$--even and, by the above, all its interior edges apart from $e$ are of degree four. But then, by Claim 1, there exists a boundary edge $f \in \partial \lst$ satisfying $d_{\lst}(f) \geq 5$. Now
$m-1$ additional maximal layered solid tori containing degree three edges meet $f$ in one of their $H$--even boundary edges. By Lemma~\ref{lem:intersection}, $f$ is of degree at least $4 + 2m.$ Thus $d(f)-4\ge 2m > m$ and we have a deficit in (\ref{eq:main}).
 
The last case is that  $\lst$ is of type $\Tqq$.
Since $\lst$ is of type $\Tqq$ it has a unique $H$-even boundary edge $f \in \partial \lst$. If $d_{\lst}(f) > 1,$ then $d_{\lst}(f) \geq 3.$ Again $m-1$ additional maximal layered solid tori containing degree three edges are incident with $f$. By Lemma~\ref{lem:intersection}, $f$ is of degree at least $2 + 2m.$ Thus $d(f) -4 \ge  2m-2.$ So if $m\ge 2$, then $d(f) -4 \ge m.$
If $m=1,$ then we claim that the degree of $f$ is at least five. Suppose it is not. Then the neighbourhood of $f$ is modelled on a layered $\partial$--punctured solid torus having one more tetrahedron than $\lst$ but extra identifications between the boundary edges. This is not possible due to maximality and Lemma~\ref{lem:lst no folds}. Hence no identifications are possible amongst the edges and so $d(f)\ge 5.$
\qed

\medskip

The set of internally unbalanced maximal layered $\partial$--punctured solid tori can now be canonically partitioned, by grouping together those which meet along their unique $H$--even boundary edge. Fix one such set of $m$ internally unbalanced maximal layered $\partial$--punctured solid tori around a common $H$--even edge $f.$ Its contribution to the right hand side of~(\ref{eq:main}) is $m.$ But $f$ must be of degree at least $2m$ by Lemma~\ref{lem:intersection}, and its contribution to the left hand side of~(\ref{eq:main}) is $\geq 2m-4.$ Hence, there is a deficit if $m > 4.$ We therefore consider the cases where $m \le 4.$ If $m=4$ and $d(f) > 8,$ we also obtain a deficit. 

\noindent {\bf Special cases}:
Since the degree $d = \operatorname{deg}(f)$ of $f$ is bounded below by four, we obtain at most a balance if: 
\[(m,d) \in \{ (4,8),(3, 7), (2,6),(1,5) \},\]
and we obtain at most a gain if: 
\[(m,d) \in \{ (3,6),(2,5),(2,4),(1,4) \}.\]
We call the edges in the above collections the \emph{balancing edges} and the \emph{gaining edges}, respectively.

The maximal potential gain is $2$ in the case $(2,4)$ and it is $1$ in all other cases. Our strategy now is to visit each edge $f$ that gives a potential gain, show that it either achieves a deficit, or (even stronger) that it achieves a deficit for itself and all $H$--even edges in its star. 

Before taking a closer look at these cases, we first prove an auxiliary statement. For this and for the remainder of this proof, we denote the tetrahedra around $f$ by $\Delta_1, \ldots , \Delta_d$ (labelled in cyclic order such that $\Delta_i$ shares a triangle with $\Delta_{i+1};$ some tetrahedron may appear in this sequence multiple times). Moreover, without loss of generality and following Lemma~\ref{lem:intersection}, the internally unbalanced maximal layered $\partial$--punctured solid torus  $\lst_i$ incident to $f$ can be assumed to contain $\Delta_{2i - 1},$ $1 \leq i \leq m.$ 

\medskip

\noindent {\bf Claim 3}: In all four cases $(m,d) \in \{ (3,6),(2,5),(2,4),(1,4) \},$ the edge $f$ is surrounded by $d$ pairwise distinct tetrahedra.

\medskip

\noindent {\bf Proof of Claim 3}: Tetrahedra contained in distinct internally unbalanced maximal layered $\partial$--punctured solid tori must be pairwise distinct. Hence the tetrahedra $\Delta_{2i - 1},$ $1 \leq i \leq m,$ are pairwise distinct since $\Delta_{2i - 1}$ is contained in $\lst_{i}.$ 

A tetrahedron of type $\Tqtt$ cannot occur twice around $f$ because it has only one $H$--even edge and, of course, there can be no tetrahedron of type $\Tqqq.$ Moreover, only $\Delta_3$ in case $(1,4)$ can be of type $\Tee,$ and hence a tetrahedron of type $\Tee$ cannot occur more than once around $f.$ 

If $f$ occurs more than once in a tetrahedron $\Delta$ of type $\Ttt,$ then we must be in $(2,5)$ or $(1,4)$ since $f$ is contained in a face only having $H$--even edges. In case $(2,5)$ we have $\Delta= \Delta_4 = \Delta_5;$ in case $(1,4)$ we either have $\Delta= \Delta_2 = \Delta_3,$ or $\Delta= \Delta_3 = \Delta_4,$ or $\Delta= \Delta_2 =\Delta_4.$ Assume we have one of the three cases where $\Delta = \Delta_i = \Delta_{i+1}.$ Then two faces of $\Delta$ are identified. However, there is only one way to identify two faces of a tetrahedron of type $\Ttt$ in an orientation preserving manner, and this creates an edge of degree one, which is not possible. Hence we are in case $(1,4)$ and $\Delta= \Delta_2 =\Delta_4.$ Then $\Delta_3$ is necessarily of type $\Tee.$ But $\Tee$ cannot share two distinct faces with $\Delta_2 = \Delta_4$ because a tetrahedron of type $\Ttt$ only has one triangle of $H$--even edges.

The remaining tetrahedron type $\Tqq$ can occur in all cases and in all places around $f.$ To start suppose $\Delta_2$ is of type $\Tqq;$ one of the $H$--even edges of $\Delta_2$ is $f$ and we denote the other $f'.$ If $\Delta_2$ appears more than once around $f,$ then $f=f'.$ In particular the two faces of $\Delta_2$ containing $f'$ meet $\Delta_1$ in at least two edges and hence cannot be identified with a tetrahedron $\Delta_{2i - 1},$ $2 \leq i \leq m.$ This shows that in the cases $(3,6)$ and $(2,4)$ all tetrahedra must be pairwise distinct by the symmetry of the situation.

In the case $(2,5)$ this implies that $\Delta_2$ cannot equal $\Delta_4.$ However, $\Delta_2$ also meets a face of $\Delta_3$ and by the same reasoning cannot equal $\Delta_5.$ The only remaining possibility in this case is that $\Delta_4$ equals $\Delta_5.$ Here again, $\Delta_5$ meets $\Delta_1$ in a face, and hence if $\Delta_4$ equals $\Delta_5,$ then $\Delta_1$ and $\Delta_3$ share two edges, which is not possible.

It remains to consider the case $(1,4).$ First suppose that $\Delta_2$ and $\Delta_3$ are the same tetrahedron. There are two possibilities that are combinatorially equivalent. Taking all the face pairings into account, we conclude that $\Delta_2$ is a 1--tetrahedron $\lst^\star(1, 2, 3).$ But this meets $\Delta_1,$ and hence $\lst_1,$ in three edges, a contradiction. See Figure~\ref{fig:distinct1} for an illustration of this case.

\begin{figure}[htb]
  \centerline{\includegraphics[width=.45\textwidth]{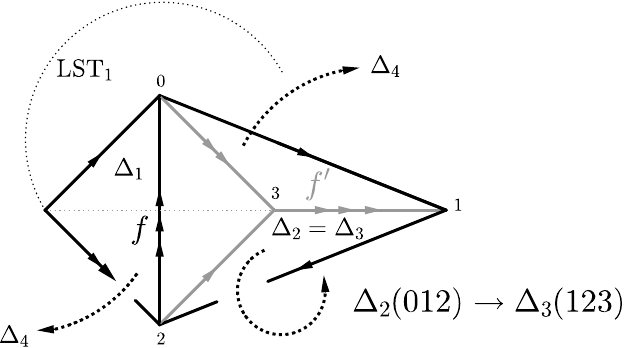}}
  \caption{Identifying $\Delta_2$ and $\Delta_3$ in case $(1,4)$ yields a one-tetrahedron layered solid torus meeting $\lst_1$ in three edges. \label{fig:distinct1}}
\end{figure}

Hence suppose without loss of generality that $\Delta_2$ and $\Delta_4$ are the same tetrahedron. Analysing the two different gluings that arise, we see that $\Delta_2$ is layered on one of the two $H$-odd boundary edges of $\Delta_1.$ But this either contradicts maximality of $\lst_1$ or it contradicts Lemma~\ref{lem:lst no folds}, see Figure~\ref{fig:distinct2}.

\begin{figure}[htb]
  \centerline{\includegraphics[width=.45\textwidth]{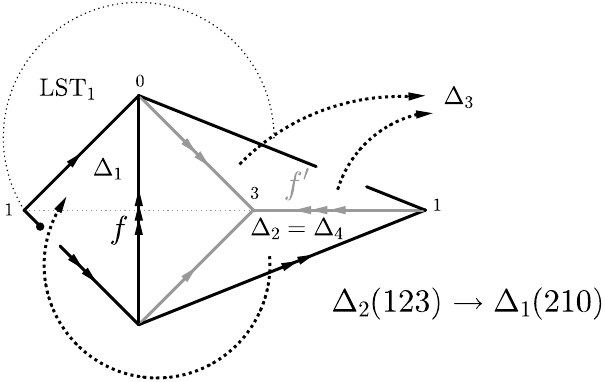}}
  \caption{Identifying $\Delta_2$ and $\Delta_4$ in case $(1,4)$ yields a layering on the boundary of $\lst_1.$ \label{fig:distinct2}}
\end{figure}

This completes the proof of the claim.\qed
 
 \medskip

\noindent {\bf The case $m=2$ and $d=4$}: Since $f$ is $H$--even of degree four, we have to find a contribution to the left hand side of~(\ref{eq:main}) of $2$ to obtain a balance and $>2$ to obtain a deficit.

By Claim 3 and by the preliminary observations above, we can assume that all four tetrahedra surrounding $f$ are distinct and tetrahedra $\Delta_1$ and $\Delta_3$ are of both of type $\Tqq.$ If the quadrilaterals in $\Delta_1$ and $\Delta_3$ are of distinct colours, $\Delta_2$ and $\Delta_4$ must both be of type $\Tqtt.$ Hence, they contribute $2$ to the left hand side of~(\ref{eq:main}). Moreover, since these two tetrahedra then do not have any $H$--even edges other than $f,$ they cannot contribute to any other set of maximal layered solid tori united along an $H$--even edge, and their contribution to~(\ref{eq:main}) is global. In particular, there is at most a balance and not a gain. Now one of the surfaces, $S_i,$ links the edge $f$ in an annulus. We can perform a compression of this annulus, using a disc transverse to $f.$ This is indeed a compression disc since the resulting surface is still connected, due to its intersection with the $H$-odd boundary edges of the layered $\partial$--punctured solid tori. Whence in total, we have a deficit of $4.$ We also note that this compression disc is confined to a neighbourhood of $f$ and that the tetrahedra $\Delta_2$ and $\Delta_4$ have a \emph{unique} $H$--even edge. Thus, the quadrilateral discs of $S_i$ affected by this compression allow no other compressions of this kind and we may associate the total deficit with $f.$

We can thus assume that $\Delta_1$ and $\Delta_3$ are coloured the same. This implies that $\Delta_2$ and $\Delta_4$ must both be of type $\Tqq$ with the same colours as well. In particular, the octahedron that is the boundary of the star of $f$ has eight $H$-odd edges, all of degree two inside the star of $f,$ and a 4--cycle of $H$--even edges, all of degree $1$ inside the star of $f.$ In this case we can perform compressions of two of the surfaces as above. Whence in total we have a contribution of $8$ to the left hand side and a contribution of $2$ to the right hand side. We now allocate a contribution of $3$ to $f$ so that in total we have a deficit of $1$ associated with $f.$ We then associate a contribution of $2.5$ to each of the other two $H$--even edges of $\Delta_2$ and $\Delta_4.$ Hence each of these (if it is in our collection of gaining or balancing edges) now has a deficit associated with it and is not visited again. In particular, any further compression discs we find in the star of an $H$--even edge is disjoint from the ones we just found.

\medskip

This completes the case of gaining edges of type $(2,4)$ showing that they all contribute a deficit.
We may thus assume that the maximal potential gain of an edge we have not yet considered is 1.

\medskip

\noindent {\bf The case $m=3$ and $d=6$}: Since $f$ is $H$--even of degree six, it contributes $2$ to the left hand side of~(\ref{eq:main}). Hence, we have to find an additional contribution of $1$ to obtain a balance and $>1$ to obtain a deficit.

Since $\lst_1,$ $\lst_2$ and $\lst_3$ are internally unbalanced, $\Delta_1,$ $\Delta_3$ and $\Delta_5$ are all of type $\Tqq.$ If not all of them are coloured the same, at least two of $\Delta_2,$ $\Delta_4$ and $\Delta_6$ must be of type $\Tqtt,$ contributing $2$ to the left hand side of~(\ref{eq:main}) and we are done (again, note that the two tetrahedra of type $\Tqtt$ do not have any other $H$--even edges, and their contribution to the star of $f$ is global). 

Hence all tetrahedra $\Delta_i,$ $1\leq i \leq 6,$ are of type $\Tqq,$ all with the same colours. This implies that one can perform a compression of two of the surfaces $S_i,$ $i  \in \{1,2,3\},$ by cutting along the two tubes around the $H$--even edge $f,$ and pasting two discs, each intersecting $f$ once (see Figure~\ref{fig:compression} for an illustration of this process). The compressed surfaces are still embedded.

Moreover, note that three of the quadrilaterals in the tubes linking edge $f$ are contained in layered solid tori $\lst_j$, $1 \leq j \leq 3$. Hence, each such quadrilateral has its vertices identified in diagonally opposed pairs.  In particular, the compressed surfaces still have the same number of connected components, and their components still have non-positive Euler characteristic and Lemma~\ref{lem:chi bounds norm for canonical} applies. Altogether, the possibility of such compressions yields a contribution of $8$ to the term $2D \geq 4 \sum_{i=1}^{3} d_i$ of the left hand side of~(\ref{eq:main}).

\begin{figure}[htb]
  \centerline{\includegraphics[width=.8\textwidth]{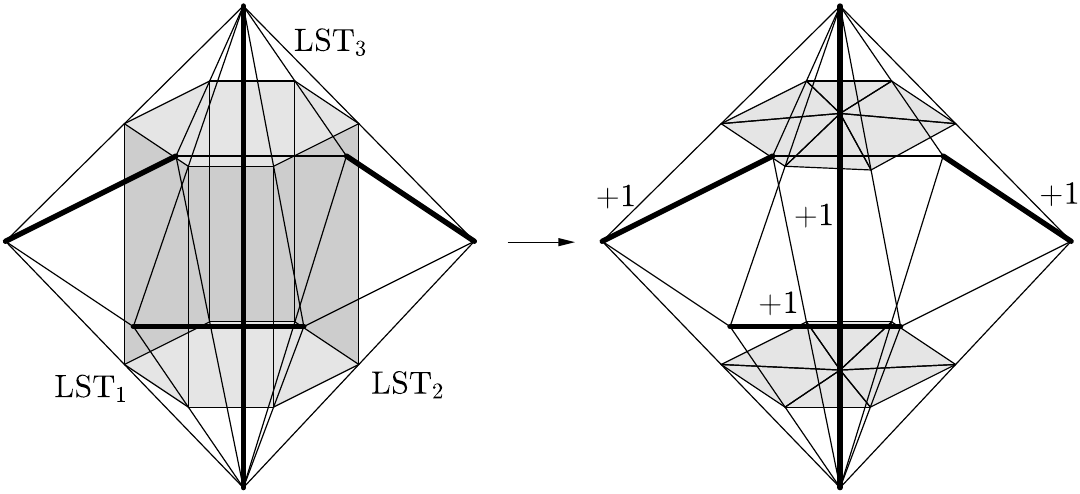}}
  \caption{Compression along an $H$--even edge surrounded by tetrahedra of type $\Tqq.$ The picture only shows one of the two parallel surfaces. \label{fig:compression}}
\end{figure}

We distribute this contribution onto $f$ and the three remaining $H$--even edges of $\Delta_2,$ $\Delta_4$ and $\Delta_6$ equally. Thus, all of the four edges receive a sufficient contribution to give a deficit in (\ref{eq:main}), and are not visited again (except that they may receive additional deficit from other edges). In particular, any other compression discs we find do not intersect the tetrahedra in the star of $f.$

This completes the case of gaining edges of type $(3,6)$ showing that they all contribute a deficit.

\medskip 

\noindent {\bf The case $m=2$ and $d=5$}: Since $f$ is $H$--even of degree five, it contributes $1$ to the left hand side of~(\ref{eq:main}). Hence, we have to find an additional contribution of $1$ to achieve a balance and $>1$ to achieve a deficit.

Again, $\Delta_1$ and $\Delta_3$ can be assumed to be of type $\Tqq.$ If $\Delta_1$ and $\Delta_3$ are of distinct colours, $\Delta_2$ must be of type $\Tqtt.$ If there is at least one other tetrahedron of type $\Tqtt$ incident with $f,$ then there is a contribution to the left hand side of~\ref{eq:main} of $2$ and we have a deficit. Otherwise there are two distinct tetrahedra of type $\Ttt$ also incident with $f.$ We obtain a contribution of $2$ to the left hand side of~(\ref{eq:main}) from these tetrahedra which needs to be shared with the two other $H$--even edges, $g$ and $h$ of the triangle shared by $\Delta_4$ and $\Delta_5.$ We can distribute a contribution of $\frac{2}{3}$ from these tetrahedra to the (possibly not pairwise distinct edges) $f,$ $g$ and $h.$ This yields a total contribution of $\frac{5}{3}>1$ to $f$ and hence a deficit.

Hence, assume that  $\Delta_1$ and $\Delta_3$ are equally coloured and thus $\Delta_2$ must be of type $\Tqq$ with matching colours. If any of $\Delta_4$ or $\Delta_5$ is of type $\Tqtt,$ the other one must be as well, we obtain a contribution of $2,$ and we have a deficit. 

If one of $\Delta_4$ and $\Delta_5$ is of type $\Tqq$ the other one must be as well. As before, this implies that one can perform a compression of two of the surfaces $S_i,$ $i  \in \{1,2,3\},$ by cutting along the two tubes around the $H$--even edge $f.$ Again, there is an additional contribution to the left hand side of $8,$ distributed equally over the four $H$--even edges of $\Delta_2,$ $\Delta_4$ and $\Delta_5$ (see Figure~\ref{fig:compression}). In particular, each of these edges has an associated deficit.

Hence, we can assume that both $\Delta_4$ and $\Delta_5$ are of type $\Ttt,$ both with the same colouring. We obtain a contribution of $2$ to the left hand side of~(\ref{eq:main}) which needs to be shared with the two other $H$--even edges, $g$ and $h$ of the triangle shared by $\Delta_4$ and $\Delta_5.$ However, for this contribution to be insufficient globally, both $g$ and $h$ must be surrounded by some internally unbalanced maximal layered solid tori, with the tetrahedra surrounding $g$ and $h$ and distinct from $\Delta_4$ and $\Delta_5$ all being of type $\Tqq$ and of matching colours. This implies that one can perform two compressions of two of the surfaces $S_i,$ $i  \in \{1,2,3\},$ by cutting along the three times two incomplete tubes around the edges $f,$ $g$ and $h.$ This process replaces the two parallel triangles inside $\Delta_4$ and $\Delta_5$ by pairs of parallel normal triangles of all three other types. Again, quadrilaterals in the tubes linking edges which are contained in a layered solid torus have their vertices identified in diagonally opposed pairs. Hence this operation does not produce any new connected components. See Figure~\ref{fig:doubleCompression} for an illustration of this process in the case of one surface (rather than two parallel ones).

\begin{figure}[htb]
  \centerline{\includegraphics[width=\textwidth]{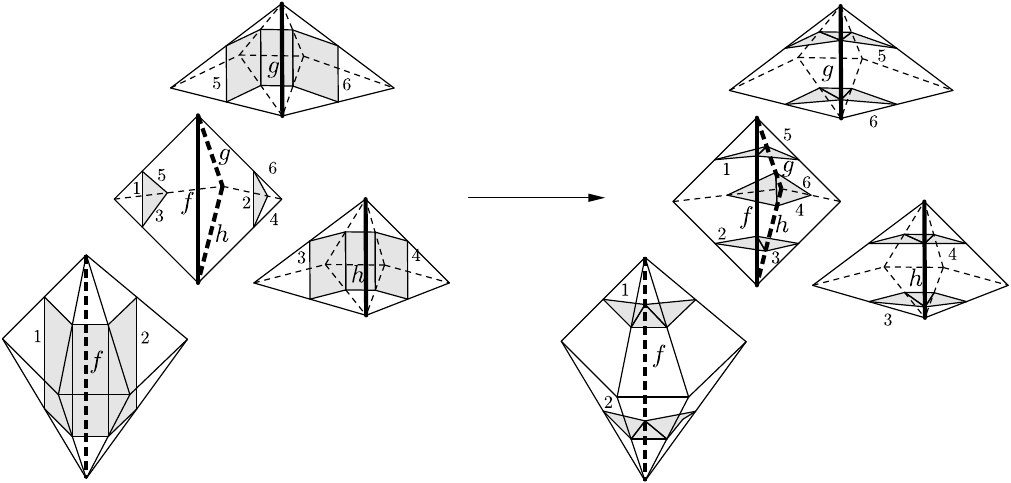}}
  \caption{Double compression along a triangle of $H$--even edges. Only one of the two parallel copies of surfaces is drawn. \label{fig:doubleCompression}}
\end{figure}

This double compression of two surfaces contributes a total of $16$ to the left hand side of~(\ref{eq:main}). We distribute two to each of $f,$ $g$ and $h,$ as well as to the at most three other $H$--even edges in the neighbourhoods of $f,$ $g$ and $h.$ The extra contribution of at least $4$ is a surplus and not needed. We therefore again have a deficit.

This completes the case of gaining edges of type $(2,5)$ showing that they all contribute a deficit.

\medskip 

\noindent {\bf The case $m=1$ and $d=4$}: In this case, $f$ does not contribute to the left hand side of~(\ref{eq:main}) and we have to find a total contribution of $1$ for a balance and $>1$ for a deficit.

Again, $\Delta_1$ can be assumed to be of type $\Tqq.$ If any of $\Delta_2,$ $\Delta_3$ and $\Delta_4$ is of type $\Tqtt$ then we again have two cases. Either at least two of these are of type $\Tqtt$ and 
we obtain a contribution of at least $2$ to the left hand side of~(\ref{eq:main}) and we are done, or one of them is of type $\Tqtt$ and two are of type $\Ttt.$ In this case we distribute the contribution of $2$ to the tetrahedra of type $\Ttt$ evenly over their three $H$--even edges and obtain a total contribution $>1$ for $f.$
If all $\Delta_i,$ $1 \leq i \leq 4$ are of type $\Tqq,$ there exist compressions for two of the surfaces $S_i,$ $i \in \{1,2,3\},$ with a total contribution of $8,$ which can be distributed equally over the four $H$--even edges contained in $\Delta_2,$ $\Delta_3$ and $\Delta_4.$ This again gives a deficit for all $H$--even edges involved.

There are two cases remaining (up to symmetry). In the first case we have (without loss of generality) $\Delta_2$ of type $\Tqq$ and $\Delta_3$ and $\Delta_4$ of type $\Ttt.$ Here, we fall back onto the double compression case along a triangle of $H$--even edges, as done in the case $m=2,$ $d=5$ (see Figure~\ref{fig:doubleCompression}). This again gives a deficit for all $H$--even edges involved.

In the second case we have $\Delta_2$ and $\Delta_4$ of type $\Ttt$ and $\Delta_3$ of type $\Tee.$ In this case, we obtain another minimal triangulation of $\manifold$ by performing a 4-4-move on $f.$ This results in two tetrahedra of type $\Ttt$ and two of type $\Tqq.$ But this contradicts our hypothesis that $\tri$ locally has the smallest number of tetrahedra of type $\Tee.$ In particular, this case does not happen. 

This completes the case of gaining edges, showing that they all contribute a deficit.
	
\medskip
	
To sum up, it is not possible to achieve a gain, and hence we have proved the first part of the theorem, establishing the basic lower bound on the number of tetrahedra.

\medskip

To prove the second part, we make some additional observations regarding the balancing edges.

If the balancing edge $f$ received a contribution $>0$ as a member of the star of a gaining edge, then it has an associated deficit. Hence assume that $f$ has not received such a contribution. If $f$ is incident with a tetrahedron of type $\Tqtt$ then there is a deficit. If $f$ is incident with a tetrahedron of type $\Ttt,$ then this tetrahedron has not been counted yet since otherwise $f$ would have received a contribution. We can therefore allocate $1/3$ to each edge incident with this tetrahedron and hence $f$ has a deficit.

\emph{In particular, this shows that associated to the subcomplex that is the union of all internally unbalanced tori and all tetrahedra of type $\Tqtt$ or $\Ttt$ we have an overall deficit.}

\medskip

These observations regarding the balancing edges suffice to analyse the case of equality. Hence suppose \(|\tri| = \sum \limits_{0\neq c \in H} ||\ c\ ||\) and so
\begin{equation*}
  \ntt+\nqtt+2 D+\sum_{d=5}^{\infty} (d-4) \even_d = \even_3.
\end{equation*}

The above analysis shows that
this equality is only possible if there is no deficit since otherwise the right hand side is smaller than the left hand side. In particular, the subcomplex that is the union of all internally unbalanced tori and all tetrahedra of type $\Tqtt$ or $\Ttt$ must be empty. Whence $\ntt = \nqtt=0.$ From the possible face pairings between tetrahedra of the different types, this implies that either all tetrahedra in $\tri$ are of type $\Tee$ (but then $H$ has rank 0); or all tetrahedra are of type $\Tqq$ (but then $H$ has rank 1); or all are of type $\Tqqq.$ Hence we have the last case and each canonical surface is a quadrilateral surface.
Since all tetrahedra are of type $\Tqqq,$ there are no $H$--even edges. Whence $\even_3=0,$ which implies $D=0$ and the three surfaces therefore are taut. 

\medskip

It remains to show that in {\em every} minimal triangulations of $\manifold$ the canonical normal representatives of non-trivial classes in $H$ are taut and all tetrahedra are of type $\Tqqq$. Recall that $\tri$ was obtained from an arbitrary minimal triangulation $\tri'$ by 4-4 moves.

We first note that every edge $e$ of $\tri$ is contained in at most $d(e) -1$ distinct tetrahedra. Otherwise the unique canonical representative $S$ disjoint from $e$ can be altered by cutting out the tube around $e$ and pasting in two discs orthogonal to $e$. Since $S$ is taut, this operation cannot be a compression of $S$ and thus it must produce two connected components -- one of which must be a normal $2$-sphere. This contradicts the fact that $\tri$ is $0$-efficient due to Theorem~\ref{lem:bdEfficient}.
But if no edge $e$ of $\tri$ is contained in $d(e)$ distinct tetrahedra, it is in particular true that $\tri$ does not admit 4-4-moves. Hence $\tri'=\tri.$ In particular, the minimality assumption on the number of tetrahedra of type $\Tee$ is true for all minimal triangulations of $\manifold$.

\medskip

We now show that if all tetrahedra are of type $\Tqqq$, then the total number of tetrahedra in $\tri$ is even. For this we use the fact that $\manifold$ is orientable and we assume that all tetrahedra are coherently oriented. In a tetrahedron of type $\Tqqq$, every boundary triangle has one 1-even, one 2-even and one 3-even edge and the orientation on the triangle induced by the orientation of the tetrahedron determines a cyclic order of the labels 1, 2, and 3. This is the same cyclic order for all boundary triangles of the same oriented tetrahedron. Hence there are two types of oriented tetrahedra of type $\Tqqq$. Tetrahedra meeting in a triangle must induce opposite cyclic orders on this triangle, see Figure~\ref{fig:even}. Hence, they are of different orientation types. It follows that the number of tetrahedra is even.
\end{proof}

\begin{figure}[htb]
  \centerline{\includegraphics[width=.6\textwidth]{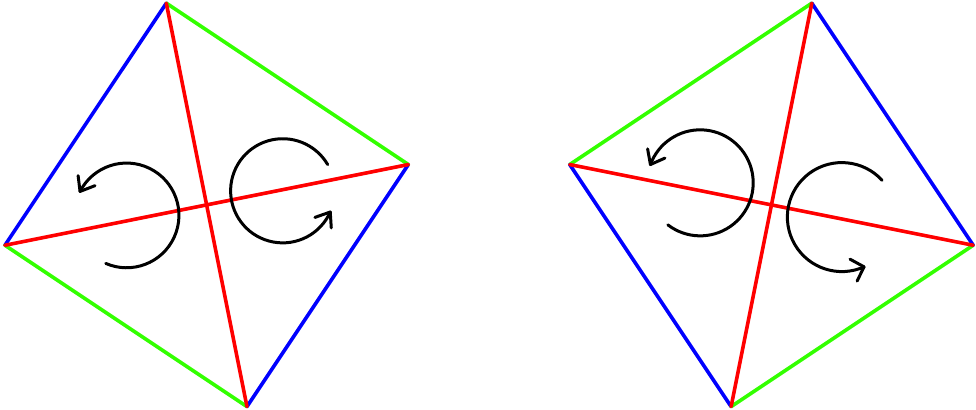}}
  \caption{The two orientation types of a tetrahedron of type $\Tqqq$. \label{fig:even}}
\end{figure}


\section{The monodromy ideal triangulations of once-punctured torus bundles}
\label{sec:monodromy ideal triangulations}

In this section we apply Theorem~\ref{thm:sumofnorms} to a nice class of examples known as the (canonical) monodromy ideal triangulations of once-punctured torus bundles over the circle. In particular, we show that these are minimal. In the last subsection we describe an infinite family of triangulations of semi-bundles that we conjecture to be minimal.


\subsection{Definition of the triangulation}
\label{sec:mit}

Monodromy ideal triangulations were studied in Lackenby~\cite{Lackenby-canonical-2003} and 
Gueritaud \cite{Gueritaud-canonical-2006}, and we refer the reader to these references for details. 

The monodromy ideal triangulations all have the property that they are ideal triangulations for which all edges are of even degree and there is a single ideal vertex. As we are only interested in the case of hyperbolic 3-manifolds amongst once punctured torus bundles, we restrict to monodromy as a map $A \in SL(2,\Z)$ having trace different from $0,\pm 1, \pm 2.$ 

Each tetrahedron in the monodromy ideal triangulation is \textit{layered} on a once punctured triangulated torus with two ideal triangles. In particular, two sets of opposite edges of each tetrahedron are identified in pairs so that the top and bottom pairs of faces form such triangulated once punctured tori. Each tetrahedron induces a diagonal flip on the triangulation of the once punctured torus, and tracing the sequence of flips gives rise to a conjugate of $\pm A.$ Indeed, each flip corresponds to 
one of the standard transvection matrices
\[ 
R = \begin{pmatrix} 1 & 1 \\ 0 & 1 \end{pmatrix} 
\quad \text{and} \quad 
L = \begin{pmatrix} 1 & 0 \\ 1 & 1 \end{pmatrix},
\]
and hence one obtains a factorisation of a conjugate of $\pm A$ in terms of positive powers of $L$ and $R.$ 

A key property of these triangulations is that the normal quadrilateral discs are of two types---called \textit{horizontal} and \textit{vertical}. A horizontal quadrilateral has vertices on the two pairs of opposite edges which are glued together. A vertical quadrilateral has two vertices on such a pair of edges and two vertices on the two edges which are not glued together. 

A crucial fact about normal surfaces in these triangulations is that the Euler characteristic is exactly the negative of the total number of horizontal quadrilaterals, see Lemma~\ref{lem:chiofS}. In particular no normal triangle and no vertical quadrilateral contributes to the Euler characteristic. 

In Sections~\ref{subsec:Reduction using coverings} and \ref{ssec:closedonesided} we prove the following result, first stated in the introduction.

\begin{reptheorem}{thm:torusbundles}
  Monodromy ideal triangulations of once-punctured torus bundles are minimal.
\end{reptheorem}


\subsection{Reduction using coverings}
\label{subsec:Reduction using coverings}

Monodromy ideal triangulations of once punctured torus bundles $\manifold$ have the nice property that any cyclic covering induced by the bundle structure is again a once punctured torus bundle and the triangulation lifts to the monodromy ideal triangulation. Throughout this section, we assume that the monodromy has trace different from $0, \pm 1, \pm 2.$ Hence there is a unique $\Z$--summand of $H_1(\manifold, \Z).$ The cyclic coverings we are using in this section come from the kernels of maps $\pi_1(\manifold) \to H_1(\manifold,\Z) \to \Z \to \Z_k.$ 

We claim that either $H_1(\manifold,\Z_2)$ has rank 3 or there is a  2-- or 3--fold cyclic covering $\widetilde{\manifold}$ of $\manifold$ with the property that $H_1(\widetilde{\manifold},\Z_2)$ has rank $3.$
The natural presentation of the fundamental group of $\manifold$ is of the form
\[\langle \ t, a, b \mid t^{-1}at = \alpha(a), t^{-1}bt = \alpha(b)\ \rangle,\]
where $\alpha$ is the automorphism of the free group $\langle a, b\rangle$ induced by the monodromy $A.$ 
It follows that the  map induced by $A$ on $\Z_2\oplus \Z_2$ is precisely its image $\overline{A} \in SL(2,2) \cong \Sym(3).$
Whence the rank of $H_1(\manifold,\Z_2)$ is 3 if $\overline{A}$ is the identity. Let $k\in \{1, 2, 3\}$ be the order of $\overline{A}.$ 
Since the monodromy of the $k$--fold cyclic covering of $\manifold$ corresponding to the kernel of the map $\pi_1(\manifold) \to \Z \to \Z_k$ gives rise to the map $\overline{A}^k \in SL(2,2),$ it follows that we obtain rank 3 after passing to a $k$--fold cyclic covering. 

It now suffices to show that a monodromy ideal triangulation is minimal for a once punctured torus bundle $\manifold$ with $H_1(\manifold, \Z_2)$ of rank $3.$ For if the rank is not three, we can pass to a 2-- or 3--fold cyclic covering $\widetilde{\manifold}$ with $H_1(\widetilde{\manifold},\Z_2)$ of rank $3$ as in the previous paragraph. If we can prove that the lifted monodromy ideal triangulation of $\widetilde{\manifold}$ is minimal then the monodromy ideal triangulation of $\manifold$ must also be minimal.


\subsection{Closed one-sided incompressible surfaces}
\label{ssec:closedonesided}

In what follows, the {\em weight} of a normal surface $S$ in a triangulation $\tri$ is defined to be the overall number of intersection points of $S$ with the edges of $\tri$.

\begin{lemma}
  \label{lem:chiofS}
  Let $\manifold$ be a once-punctured hyperbolic torus bundle and let $\tri$ be a monodromy ideal triangulation of $\manifold.$ Furthermore, let $S \subset \tri$ be a normal surface of $\tri$ which is a taut least weight representative of an even torsion class of $H_2 (\manifold , \Z).$ Then $S$ meets each tetrahedron of $\tri$ in a single quadrilateral. Moreover, $\chi(S) = - (\# \textrm{ horizontal quads}).$
\end{lemma}

\begin{proof}
  Let $S \subset \tri$ be a normal surface of $\tri$ which is a taut representative of a fixed even torsion class $c \in H_2 (\manifold , \Z).$ We assume that $S$ is of least weight amongst all taut representatives. The monodromy ideal triangulation $\tri$ naturally endows $\manifold$ with a circular Morse type function: The layering defines a cyclic ordering of pairs of faces forming a once-punctured torus fibre. Let $T$ be such a two-triangle fibre. 

Consider the intersection $S \cap T.$ This is a family of embedded disjoint normal simple closed curves in the two-triangle structure on $T.$ A well-known elementary fact is that any such curve is either non separating or vertex linking. Next, note that $S \cap T$ cannot contain two parallel curves, since otherwise we can perform an annular compression using an innermost fibre $0$ weight annulus $A$ and find a new homologous normal surface of less weight or less genus or both. 

To be more specific, consider $N(A),$ where $N(A)$ is a small regular neighbourhood of $A.$ Then $S \cap N(A)$ contains two annuli and these are replaced by two annuli in $\partial N(A)$ to form a new surface with the same weight and Euler characteristic as $S$  and representing the same class $c.$ If this surface is disconnected, it is easy to see one component is one-sided and the other is two-sided, as $c$ is a torsion class. Hence we can discard the two-sided component (it must bound in $\manifold$) and the resulting surface still represents $c.$ As this new surface is not normal, it can be isotoped to a normal surface of less weight, giving a contradiction. 

Hence, from now on, we can assume that each fibre meets $S$ in one or two curves. Note that in the latter case exactly one curve links the ideal vertex in the fibre and the other curve is non separating. Let $\gamma$ denote the non separating curve in $S \cap T.$

The curve $\gamma \subset T$ is determined by how often it intersects the three edges of $T.$ Since $\gamma$ is connected, these three intersection numbers are pairwise coprime. This can be seen by the fact that, given two intersection numbers, the third one must either be the sum or the non-negative difference of the first two. 

  A key observation is that $S$ propagates across the tetrahedron $\Delta \in \tri$ layered on top of $T$ in a very well-behaved way. In what follows, we denote the edge of $\Delta$ opposite $T$ by $e,$ and the new fibre opposite $T$ (i.e., the two triangles of $\Delta$ containing $e$) by $T'.$

  Assume that $S \cap T$ contains a vertex linking curve $\alpha.$ $\alpha$ intersects $T$ in six normal arcs (one of each type). Four of them necessarily yield two normal triangles in $\Delta$ and two normal arcs in $T'.$  For the remaining two arcs, three cases need to be considered.

  1) They are in the boundary of a quadrilateral normal disc in $\Delta.$ In this case, the two normal triangles and the normal quadrilateral containing $\alpha$ intersect both triangles of $T'$ in two parallel normal arcs. Moreover, all other normal arcs of $S \cap T'$ must be of the same type (otherwise $S$ self-intersects inside $\Delta$). In particular, $S \cap T'$ has parallel curves, which contradicts our assumption that $S$ is least weight, using the annular compression argument from above. See Figure~\ref{fig:toruslevel_2} on the left.

  2) One of them yields a normal triangle $t,$ and the other one is in the boundary of a normal quadrilateral in $\Delta$ together with a normal arc of $\gamma.$ In this case, $S \cap T',$ again consists of one type of normal arc per triangle of $T',$ with two exceptional normal arcs coming from $t.$ Let one triangle of $T'$ contain $k$ and the other one contain $\ell$ copies of parallel normal arcs. Hence, one edge of $T'$ intersects $S$ once, the second one $k+1 = \ell$ times and the third $k$ times, hence, $\ell + 1 = k,$ a contradiction. See Figure~\ref{fig:toruslevel_2} on the right.

  3) All of the six normal arcs of $\alpha$ are in the boundaries of normal triangles in $\Delta.$ Since 1) and 2) are not possible, this must be the case for all tetrahedra of $\tri$ and $S$ contains a boundary parallel torus. This contradicts the assumption that $S$ is a least weight taut representative of $c.$

\begin{figure}[htb]
 \centerline{\includegraphics[width=.85\textwidth]{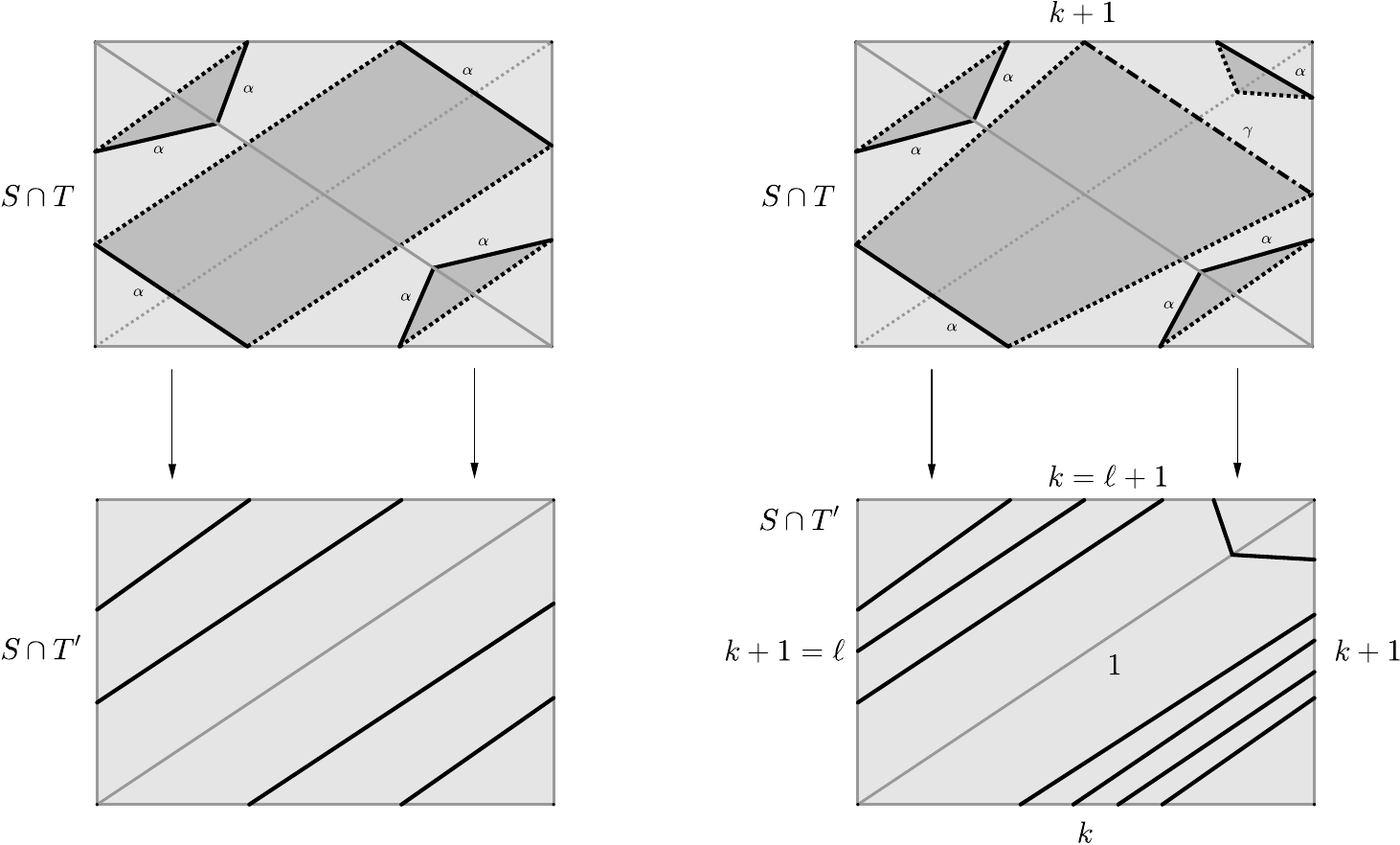}}
 \caption{Left: The normal arcs of $\alpha$ are in two normal triangles and one normal quadrilateral. Right: The normal arcs of $\alpha$ are in three normal triangles and one normal quadrilateral. \label{fig:toruslevel_2}}
\end{figure}

Hence, for the rest of the proof we can assume that $S \cap T$ is connected for each fibre $T.$ Consequently, it follows that $S \cap T$ has only one or two normal arc types in each triangle of $T,$ since otherwise $S \cap T$ contains a vertex linking component.

We claim that $\Delta$ can only contain vertical normal discs of $S,$ unless $\Delta \cap S$ consists of a single horizontal quad and $S$ intersects two edges of $T$ exactly once, and is disjoint from $e.$

 More precisely, there are two cases we have to consider:

  \begin{enumerate}

    \item The intersection of $S$ with $T$ is disjoint from one of the edges of $T.$ In this case the other two edges must intersect $S$ exactly once. If we layer on the edge disjoint from $S$ then there is a single horizontal quadrilateral in the tetrahedron, if we layer on one of the other edges, there is a single vertical quadrilateral. See Figure~\ref{fig:toruslevel} on the right.

    \item The surface $S$ intersects $T$ in two normal arc types per triangle. In this case all normal discs of $S$ in $\Delta$ must necessarily be vertical. To see this note that a horizontal normal quadrilateral can only have two normal arcs of $S \cap T$ in its boundary. Since by the assumption, at least one other normal disc must exist in $\Delta,$ this must be a vertical normal triangle. Arguing as in case 1) and 2) above we can then follow that all three normal arc types must be present per triangle, a contradiction. See Figure~\ref{fig:toruslevel} on the left.

  \end{enumerate}

 It is now easy to verify that the Euler characteristic of $S$ is the negative of the number of horizontal quadrilateral discs, using the circular Morse type function induced by the layering. For each horizontal quadrilateral disc induces a simple saddle type singularity of Morse index one, whereas there are no critical points on vertical triangles or quadrilaterals. 

\begin{figure}[htb]
 \centerline{\includegraphics[width=.85\textwidth]{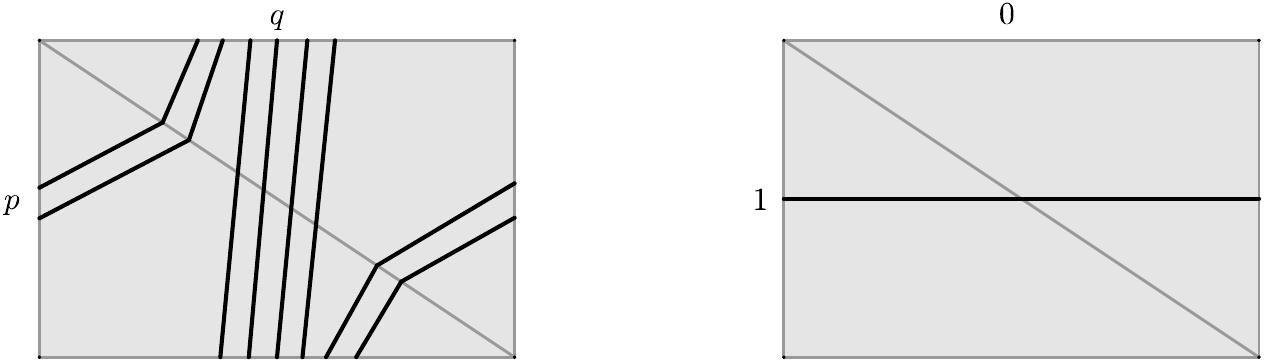}}
 \caption{Left: $S\cap T$ has two normal arc types per triangle of $T.$ Right: $S\cap T$ is disjoint from one edge of $T.$\label{fig:toruslevel}}
\end{figure}

Since $S$ is taut, no component can be a torus or Klein bottle. Hence there must be at least one horizontal quadrilateral. We can cut $\manifold$ open along a corresponding fibre $T$ to form a product of the form $T \times I$ where $T$ is a once punctured torus. This product admits an induced ideal triangulation which is minimal layered. 

There are two steps to complete the proof. The first is to follow the argument in \cite[Theorem 12]{Rubinstein-one-sided-1978} which shows that there is a unique isotopy class of incompressible surfaces in a lens space with even fundamental group. We want to apply a similar argument to an incompressible surface $S^*$ in a product of a once punctured torus and an interval. 

Consider a product $X= T \times I$ where $T$ is a once punctured torus and there are two essential boundary curves  $C, C'$ in $T \times \{0\}, T \times \{1\}$ respectively and the intersection number of $C,C'$ is even when projected to a copy of $T.$ We claim there is a unique incompressible surface $S^*$ up to isotopy with $\partial S^* = C \cup C'.$ 

Pick an annulus $A$ in the product structure of $X$ with boundary curve on  $T \times \{1\}$ parallel to $C'.$ We can assume without loss of generality that $A$ is transverse to $S^*$ and the intersection $A \cap S^*$ contains only arcs which are essential on $S^*$ and have both ends on $T \times \{0\}.$ An innermost such an arc on $A$ cuts off a bigon which can be used to perform a boundary compression of $S^*$ across $T \times \{0\}.$ Exactly as in \cite{Rubinstein-one-sided-1978}, this boundary compression must change the boundary curve $C$ of $\partial S^*$ to a new curve which has smaller intersection number with $C'.$ It is now easy to show as in \cite{Rubinstein-one-sided-1978} that this process produces a unique family of bands attached to a collar about $C$ forming $S^*$ which is therefore unique up to isotopy. 

The second step is to observe that the canonical quadrilateral surface in the induced ideal layered triangulation of the product is incompressible. The proof again follows \cite{Rubinstein-one-sided-1978}. The key idea is to use induction and the recursion formula for the genus of the incompressible surface. By removing the tetrahedra up to the first containing a horizontal quadrilateral, we obtain a new product with an induced minimal layered triangulation. The recursion formula for the change in boundary slope and the genus matches the formula in \cite{Rubinstein-one-sided-1978} and this shows that the canonical quadrilateral surface must be incompressible as it is a minimal genus cobordism surface connecting $C$ and $C'.$ This completes the proof. 
\end{proof}

\begin{corollary}
  \label{cor:rank2minimalbundle}
  Let $\manifold$ be a once-punctured hyperbolic torus bundle and let $\tri$ be a monodromy ideal triangulation of $\manifold.$ Suppose that $H_2 (\manifold , \Z)$ contains a Klein $4$-group $H.$ Then 
  $$ |\tri| = \sum \limits_{0 \neq c \in H} || \ c\ ||, $$
  and in particular $\tri$ is a minimal triangulation. 
\end{corollary}

\begin{proof}
  For each of the three non-trivial classes in $H$ we have least weight taut normal representatives. 
  We know from the first part of Lemma~\ref{lem:chiofS} that each of them meets each tetrahedron of $\tri$ in exactly one normal
  quadrilateral. Moreover, since $H$ is a Klein $4$-group, the sum of all three representatives
  must have edge weight zero mod two. It follows that every tetrahedron contains all three quadrilateral
  types exactly once. In particular, in each tetrahedron exactly one of the three representatives has a 
  single horizontal normal quadrilateral. 

  Hence, by the second part of Lemma~\ref{lem:chiofS} the negative of the sum of Euler characteristics of these three representatives must equal the number of tetrahedra of $\tri.$
  It now follows from Theorem~\ref{thm:sumofnorms} that $\tri$ is minimal. 
\end{proof} 

Combining Corollary~\ref{cor:rank2minimalbundle} with the covering argument provided in \S\ref{subsec:Reduction using coverings} proves Theorem~\ref{thm:torusbundles}.


\subsection{Further examples}
\label{subsec:more examples}

The orientable cusped hyperbolic census up to nine tetrahedra, described and verified in \cite{Burton14Census} and partially shipped with \texttt{Regina}~\cite{regina}, contains $162\,182$ minimal triangulations of $61\,911$ cusped hyperbolic $3$-manifolds. There are exactly 26 triangulations for which the lower bound in Theorem~\ref{thm:sumofnorms} is attained. Of these, $22$ are triangulations of once-punctured torus bundles, but the remaining four are not. Their isomorphism signatures are 
\begin{itemize}
  \item \texttt{gLLMQbeefffehhqxhqq} of manifold \texttt{s781}, 
  \item \texttt{iLLLQPcbefgffhhhxxhaqxxqh} of manifold \texttt{t05624},
  \item \texttt{iLLLQPcbefgffhhhhhqaxhhxq} of manifold \texttt{t06056},
  \item \texttt{iLLwQPcbeefgehhhhhqhhqhqx} of manifold \texttt{t12546}.
\end{itemize}

These examples, as well as infinite families of minimal triangulations they are contained in, will be described in detail in a separate note.


\bibliographystyle{plain}
\bibliography{minimal}


\address{William Jaco\\Department of Mathematics, Oklahoma State University, Stillwater, OK 74078-1058, USA\\{jaco@math.okstate.edu}\\-----}

\address{J. Hyam Rubinstein\\School of Mathematics and Statistics, The University of Melbourne, VIC 3010, Australia\\
{joachim@unimelb.edu.au}\\----- }

\address{Jonathan Spreer\\School of Mathematics and Statistics F07, The University of Sydney, NSW 2006 Australia\\ 
{jonathan.spreer@sydney.edu.au\\-----}}

\address{Stephan Tillmann\\School of Mathematics and Statistics F07, The University of Sydney, NSW 2006 Australia\\
{stephan.tillmann@sydney.edu.au}}

\Addresses

\end{document}